\documentclass[11pt]{amsart}
\usepackage{color}
\usepackage{stmaryrd}
\usepackage{cases}
\usepackage{amssymb}
\usepackage{amsmath}
\usepackage{amsfonts}
\usepackage{graphicx}
\usepackage{amsmath,amstext,amsbsy,amssymb, color}
\usepackage[colorlinks=true, linkcolor=blue, citecolor=magenta, urlcolor=blue, backref=page]{hyperref}

\hypersetup{colorlinks,citecolor=blue,filecolor=black,linkcolor=blue,urlcolor=blue}
\usepackage{latexsym}
\usepackage{amsmath}
\usepackage{amssymb}
\usepackage{bm}
\usepackage{graphicx}
\usepackage{wrapfig}
\usepackage{fancybox}
\usepackage{ulem}

\newtheorem{theorem}{Theorem}[section]
\newtheorem{lemma}[theorem]{Lemma}
\newtheorem{proposition}[theorem]{Proposition}
\newtheorem{corollary}[theorem]{Corollary}
\theoremstyle{definition}

\theoremstyle{remark}
\newtheorem{remark}[theorem]{Remark}

\numberwithin{equation}{section} \errorcontextlines=0

\newcommand{\diag}{\mbox{diag}}

\newcommand{\ZZ}{\mathbb Z}

\newcommand{\Imm}{\mathrm{Imm}}

\newcommand{\ot}{\otimes}

\newcommand{\si}{\sigma}
\newcommand{\GL}{\mathrm{GL}}
\newcommand{\SL}{\mathrm{SL}}
\newcommand{\gl}{\mathfrak{gl}}

\newcommand{\Mat}{\mathrm{Mat}}
\newcommand{\End}{\mathrm{End}}
\newcommand{\Ber}{\mathrm{Ber}}

\newcommand{\Ec}{\mathcal{E}}
\newcommand{\Tc}{\mathcal{T}}
\newcommand{\bal}{\begin{aligned}}
\newcommand{\eal}{\end{aligned}}
\newcommand{\beq}{\begin{equation}}
\newcommand{\eeq}{\end{equation}}
\newcommand{\ben}{\begin{equation*}}
\newcommand{\een}{\end{equation*}}
\newcommand{\Sym}{\mathfrak S}
\newcommand{\CC}{\mathbb{C}}

\begin{document}
	
	\title[Super-immanants and  Littlewood correspondences]
	{Super-immanants and Littlewood correspondences}
	
	\author{Naihuan Jing}
	\address{Department of Mathematics, North Carolina State University, Raleigh, NC 27695, USA}
	\email{jing@ncsu.edu}
	
		\author{Yinlong Liu}
	\address{Department of Mathematics, Shanghai University, Shanghai 200444, China}
	\email{yinlongliu@shu.edu.cn}
	
	\author{Jian Zhang}
	\address{School of Mathematics and Statistics,
		Central China Normal University, Wuhan, Hubei 430079, China}
	\email{jzhang@ccnu.edu.cn}
	
	\thanks{{\scriptsize
			\hskip -0.6 true cm MSC (2020): Primary: 17A70 Secondary: 17B35, 15A15, 05E05, 05E10
    }}
	

\begin{abstract}
In this paper, we introduce the notion of super-immanants for supermatrices  over a supercommutative algebra. Using the super Schur-Weyl duality we show that the super immanants play a significant role in
covariant tensor representations of the general linear Lie superalgebra. Among various things, we obtain a supertrace formula for super-immanants, which generalizes Kostant's trace formula to the super setting. Furthermore, we show that the Littlewood correspondences between super-immanants and supersymmetric polynomials establish an isomorphism between their corresponding algebras.
\end{abstract}
	\maketitle
\section{Introduction}
Supermatrices over a supercommutative algebra have been used in the study of supermanifolds, Lie supergroups, and Lie superalgebras \cite{Ber,Kos,Kac77,Le,KN84}. Their introduction was closely related to various physical problems \cite{Man, DM}.

The classical notion of immanants can be traced back to the seminal work of Schur \cite{S}. Littlewood \cite{Li, LR}
introduced these matrix functions as generalizations of determinants and permanents to formulate the Littlewood correspondences, which establish
a complete correspondence between representations of the general linear group and the symmetric group
at the basis level comparable to the Schur-Weyl duality. The Littlewood correspondences not only supplement the Schur-Weyl duality but also
reveal the close relationship with 
Gelfand-Tsetlin bases.
In this paper, using the primitive idempotents of the symmetric group, we study the super-immanants of supermatrices  associated with arbitrary representations of the symmetric group, which will provide a practical means to completely generalize the quantum Okounkov immanants.

One of the main results of this paper is to interpret super-immanants and to derive a supertrace formula via weight spaces of covariant tensor representations of general linear Lie superalgebra  $\gl_{m|n}$. In our previous work \cite{JLZ}, we upgraded Kostant's trace formula \cite{Ko} from
the $0$-weight space to arbitrary weight spaces for the special linear group $\SL(n,\CC)$ in the context of the  quantum coordinate algebra.

In this paper, we extend the approach in \cite{JLZ} by constructing  an exact correspondence between the Gelfand-Tsetlin type bases of the irreducible covariant tensor representations of $U(\gl_{m|n})$ and Young's orthonormal basis of irreducible representations of the symmetric groups.

It is well known that Schur symmetric polynomials are the characters of the irreducible representations of the symmetric groups and the general linear groups. Due to the close connection between Schur polynomials and immanants, known as the Littlewood correspondences \cite{Li},
many immanant identities can be deduced from representation theory of the symmetric and the general linear groups. On the other hand, the immanant identities provide exact numerical
relations in  the representation theory  in a spirit similar to the Schur-Weyl duality.

Thus, the second purpose of this paper is to construct the correspondences between super-immanants and  Schur supersymmetric polynomials, generalizing the Littlewood correspondences I-III to superalgebras. We also introduce a commutative subalgebra of the coordinate superalgebra  of  supermatrices and show that it is isomorphic to the algebra of supersymmetric polynomials.

It is noteworthy that Littlewood correspondence III  says that the Schur polynomials, with variables
 are the characteristic roots of the matrix $X$, can be explicitly expressed by the corresponding normalized immanants of suitably chosen generalized submatrices of  $X$.
 It is quite different in the setting of supermatrices. The characteristic functions (Berezinian) and eigenvalues of supermatrices were developed in \cite{KN88,KN90}.  Urrutia and  Morales  \cite{UM,UM2} proved the super Cayley-Hamilton Theorem and another version of super Cayley-Hamilton Theorem was found by \cite{KT1,KT2}. Later, Gurevich, Pyatov and Saponov \cite{GPS1,GPS2} generalized the result to the quantum super Cayley-Hamilton Theorem.
 Using  these techniques of characteristic functions (Berezinian) of supermatrices and the super Cayley-Hamilton Theorem,
 we obtain exactly a super analogue of Littlewood correspondence III.

As applications, we obtain super ananlogs of several classical immanant identities from the super Littlewood correspondences. In particular, we obtained the
 super versions of the Littlewood--Merris-Watkins identities \cite{Li,MW} and the Goulden-Jackson identities \cite{GJ}. Notably, the Goulden-Jackson identities  \cite{GJ} also give new interpretations of the Littlewood correspondences, MacMahon theorem \cite{Mac,VJ}, and Young idempotents.

The paper is organized as follows. In Section 2, we introduce the notion of super-immanants and give a representation-theoretic interpretation for super-immanants.
Section 3 is devoted to the super analogs of Littlewood correspondences I-III. We also show that the super Littlewood--Merris-Watkins identities and the super Goulden-Jackson identities.

\section{Super-immanants and super Schur-Weyl duality}
\subsection{Super-immanants}
Let $\CC^{m|n}$ be the complex super vector space spanned by its $\ZZ_2$-graded basis $\{e_1,\ldots,e_{m+n}\}$. For any $1\leq i\leq m+n$, the parity of $e_i$ is defined as
\[
\bar{i}=\left\{\begin{array}{cc}
0&\text{ if }1\leq i\leq m,\\
1&\text{ if }m< i\leq m+n.
\end{array}\right.
\]

Let $R=R_0\oplus R_1$ be a free supercommutative algebra over $\CC$ (such as a Grassmann algebra).
We donote by $\End(R)=\End(R)_0\oplus \End(R)_1$ the set of $\mathbb Z_2$-graded endomorphisms of $R^{m|n}:=R\ot \CC^{m|n}$.

In terms of the $\ZZ_2$-graded  basis of $\CC^{m|n}$, the even superalgebra $\End(R)_0$ can be identified with
all even square supermatrices over $R$ with dimension $(m+n)\times (m+n)$, thus we also write it as $\Mat_{m|n}(R)$. The group of
all even automorphisms of $R^{m|n}$ will be denoted by $\GL_{m|n}(R)$.
The coordinate superalgebra $A(\Mat_{m|n})$ is generated by $x_{ij}, 1 \leq i,j \leq m+n$, where the $x_{ij}$ are the coordinate functions on $\Mat_{m|n}(R)$ with parity $\bar i +\bar j$.

The supertrace $str$ of a supermatrix $Y=(y_{ij})$ is
\[
str(Y)=\sum_{i=1}^{m+n}(-1)^{\bar{i}}y_{ii}.
\]
For any  $a\in\{1,2,\ldots,r\}$ we will denote by $str_a$ the corresponding supertrace on the
superalgebra $\mathrm{End}((\mathbb{C}^{m|n})^{\ot r})$ which acts as $str$ on the $a$-th copy of $\mathrm{End}(\mathbb{C}^{m|n})$ and as the identity map on all
the other tensor factors.

It will be convenient to use a standard notation for the matrix coefficients $A_{j_1,\dots,j_r}^{i_1,\dots,i_r}$ of an even operator $A\in R\ot \mathrm{End}((\mathbb{C}^{m|n})^{\ot r})$ acting on the standard basis of  $(\mathbb{C}^{m|n})^{\ot r}$.
We denote
\begin{equation}
  A=\sum_{I,J}A_{j_1,\dots,j_r}^{i_1,\dots,i_r}\ot e_{i_1j_1}\ot \cdots \ot e_{i_rj_r},
\end{equation}
where $e_{ij}$ be the standard basis of $\mathrm{End}(\mathbb{C}^{m|n})$.
Let $\{e_{i}^*\mid 1\leq i\leq m+n\}$ be the basis of $\mathbb (\mathbb C^{m|n})^*$ dual to the basis $\{e_{i} \mid 1\leq i\leq m+n\}$ of $\mathbb C^{m+n}$, then we can write the bases of $(\mathbb C^{m|n})^{\ot r}$  and its dual bases of $((\mathbb C^{m|n})^{\ot r})^{*}$  respectively:
\beq
\mid i_1,\dots,i_r\rangle =e_{i_1}\otimes\cdots \otimes e_{i_r}, \qquad \langle i_1,\dots ,i_r\mid=(e_{i_1}\otimes\cdots \otimes  e_{i_r})^*,
\eeq
with parity $\bar I=\bar i_1+\ldots +\bar i_r$.
Then the matrix coefficients of the even  operator $A$ are given by
\beq
A_{j_1,\ldots,j_r}^{i_1,\ldots,i_r}=(-1)^{\gamma(I,J)}\langle i_1,\ldots ,i_r\mid A \mid j_1,\ldots,j_r\rangle,
\eeq
where
$\gamma(I,J)=\sum_{a} \bar i_a(\bar j_a+1)+\sum_{a<b}  \bar j_b(\bar i_a +\bar j_a)$. In particular,
\beq
\bal
A_{i_1,\ldots,i_r}^{i_1,\ldots,i_r}&=\langle i_1,\ldots ,i_r\mid A \mid i_1,\ldots,i_r\rangle,\\
str_{1,\ldots,r}(A)&=\sum_{I}(-1)^{\bar I}\langle i_1,\ldots ,i_r\mid A \mid i_1,\ldots,i_r\rangle.
\eal
\eeq

Let $\Sym_r$ be the symmetric group of degree $r$. There is a representation $\rho$ of $\Sym_r$ on $(\CC^{m|n})^{\ot r}$ given by
\beq
\si \cdot e_{i_1}\ot \cdots \ot e_{i_r}=(-1)^{\bar{I}_{\si}}e_{i_{\si^{-1}(1)}}\ot \cdots \ot e_{i_{\si^{-1}(r)}},
\eeq
for any  $\si \in \Sym_r$, where
\ben
\bar{I}_{\si}=\sum_{k<l, \atop \si(k)> \si(l)}\bar i_k \bar i_l.
\een
For any $1 \leq a<b \leq r$,  we define
\beq\label{P}
P_{(a,b)}=\sum_{i,j=1}^{m+n}(-1)^{\bar{j}}1^{\ot (a-1)} \ot e_{ij}\ot 1^{\ot (b-a-1)}\ot e_{ji} \ot 1^{\ot (r-b)},
\eeq
Then $\rho((i,i+1))=P_{(i,i+1)}$ and we denote $\rho(\si)=P_{\si}$.

Let  $I=(i_1,\ldots, i_k)$ and $J=(j_1, \ldots ,j_k)$ be two multisets ($k$-tuples) of $[m+n]$.   We denote by $X^I_J$ the generalized submatrix of $X$ whose row (resp. column) indices belong to  $I$ (resp. $J$). If $I=J$, we briefly write $X_I$.

Let $V$ be any representation of $\Sym_r$ with the character $\chi ^{V}$.
The {\it super-immanant} of $X^I_J$ associated to the representation $V$ is defined by
\begin{equation}\label{imm-def}
  \Imm_{\chi ^{V}}(X^I_J)=(-1)^{\sum_{k=1}^r\bar i_k\bar j_k}\langle i_{1},\ldots ,i_{r}\mid \chi^{V}X_1\cdots X_{r} \mid j_{1},\ldots,j_{r}\rangle.
\end{equation}

A partition $\lambda = (\lambda_1 \geq \ldots \geq \lambda_{m+n})$ is a non-increasing sequence of nonnegative integers. If $\lambda_1 + \cdots + \lambda_{m+n}=r$, we say that $\lambda$ is a partition of $r$, denoted as $\lambda \vdash r$. Denote by $H(m,n)$ the set of partitions $\lambda$ such that $\lambda_{m+1}\leq n$ and let
\begin{equation}\label{e:partitionH}
H(m,n;r)=\{\lambda\in H(m,n)\mid \lambda \vdash r\}.
\end{equation}
The set of standard Young
tableaux $\Tc$ of shape  $\lambda$ is denoted by $\mathrm{SYT}(\lambda)$. Let $c_k(\Tc)=j-i$ be the content of the  box $(i,j)$ of $\lambda$ occupied by $k$ in $\Tc\in \mathrm{SYT}(\lambda)$. And write $d_i(\Tc)=c_{i+1}(\Tc)-c_i(\Tc)$.

\begin{proposition}\label{imm-char}
  Let $I=(i_1, \ldots , i_r)$ be an ordered  multiset of $[m+n]$ and $X$ be the generator supermatrix of $A(\Mat_{m|n})$. Then for any $\lambda\vdash r$,
  \begin{equation}\label{imm-XI}
    \Imm_{\chi^{\lambda}}(X_I)=(-1)^{\bar I}\sum_{\si\in \Sym_r}\langle i_{\si(1)},\ldots, i_{\si(r)}\mid  \mathcal{E}_{\Tc}^{\lambda} X_1\cdots X_{r} \mid i_{\si(1)},\ldots, i_{\si(r)}\rangle,
  \end{equation}
where $\mathcal{E}_{\Tc}^{\lambda}$ be the primitive idempotent   of $\Sym_r$ associated with a  standard Young tableau ${\Tc}$ of shape $\lambda$.
 Moreover when $\lambda \notin H(m,n;r)$, $\Imm_{\chi^{\lambda}}(X_I)=0$.
\end{proposition}
\begin{proof}
Since $\chi^{\lambda}=\sum_{\si \in \Sym_r}\si \mathcal{E}_{\Tc}^{\lambda} \si^{-1}$,  by definition we have that
\begin{equation}
\bal
  \Imm_{\chi^{\lambda}}(X_I)&=(-1)^{\bar I}\sum_{\si\in \mathfrak{S}_r}\langle i_{1},\ldots ,i_{r}\mid P_{\si} \mathcal{E}_{\Tc}^{\lambda}P_{\si^{-1}}X_1\cdots X_{r} \mid i_{1},\ldots,i_{r}\rangle\\
  &=(-1)^{\bar I}\sum_{\si\in \mathfrak{S}_r}\langle i_{1},\ldots ,i_{r}\mid P_{\si} \mathcal{E}_{\Tc}^{\lambda}X_1\cdots X_{r} P_{\si^{-1}} \mid i_{1},\ldots,i_{r}\rangle.
 \eal
\end{equation}
Note that
\begin{align*}
P_{\si^{-1}}  \mid i_{1},\ldots,i_{r} \rangle=(-1)^{\bar I_{\si^{-1}}}\mid i_{\si(1)},\ldots, i_{\si(r)} \rangle,\\
\langle i_{1},\ldots ,i_{r}\mid P_{\si}=(-1)^{\bar I_{\si^{-1}}}\langle i_{\si(1)},\ldots, i_{\si(r)}\mid.
\end{align*}
Therefore,
\begin{equation*}
\bal
  \Imm_{\chi^{\lambda}}(X_I)&=(-1)^{\bar I}\sum_{\si\in \mathfrak{S}_r}\langle i_{1},\ldots ,i_{r}\mid P_{\si} \mathcal{E}_{\Tc}^{\lambda}X_1\cdots X_{r} P_{\si^{-1}} \mid i_{1},\ldots,i_{r}\rangle\\
  &=(-1)^{\bar I}\sum_{\si\in \Sym_r}\langle i_{\si(1)},\ldots, i_{\si(r)}\mid  \mathcal{E}_{\Tc}^{\lambda} X_1\cdots X_{r} \mid i_{\si(1)},\ldots, i_{\si(r)}\rangle.
\eal
\end{equation*}

If $\lambda \notin H(m,n;r)$, that is $\lambda_{m+1}>n$. Take the standard Young tableau $\Tc_0$ obtained by numbering the boxes by rows downwards, from left to right in every row. Then
\beq
\bal
\Imm_{\chi^{\lambda}}(X_I) &=(-1)^{\bar I}\sum_{\si\in \Sym_r}\langle i_{\si(1)},\ldots, i_{\si(r)}\mid  \mathcal{E}_{\Tc_0}^{\lambda} X_1\cdots X_{r} \mid i_{\si(1)},\ldots, i_{\si(r)}\rangle\\
&=(-1)^{\bar I}\sum_{\si\in \Sym_r}(\mathcal{E}_{\Tc_0}^{\lambda} X_1\cdots X_{r})_{i_{\si(1)},\ldots, i_{\si(r)}}^{i_{\si(1)},\ldots, i_{\si(r)}}\\
&=(-1)^{\bar I}\sum_{\si,\tau\in \Sym_r} (\mathcal{E}_{\Tc_0}^{\lambda})_{i_{\tau(1)},\ldots, i_{\tau(r)}}^{i_{\si(1)},\ldots, i_{\si(r)}} x_{i_{\tau(1)},i_{\si(1)}} \cdots x_{i_{\tau(r)},i_{\si(r)}}.
\eal
\eeq
Hence it is sufficient to prove that for  $I=(i_{1}\leq \ldots \leq  i_{r})$ and any $\tau \in \Sym_r$
\beq\label{suff-con}
\mathcal{E}_{\Tc_0}^{\lambda}\cdot e_{i_{\tau(1)}}\ot \cdots \ot e_{i_{\tau(r)}}=0.
\eeq

First  we show that
\beq
\mathcal{E}_{\Tc_0}^{\lambda}\cdot e_{i_1}\ot \cdots \ot e_{i_r}=0.
\eeq
Since $\lambda_{m+1}>n$ but there are at most $m$ different even factors  and $n$ different odd factors in $e_{i_1}\ot \cdots \ot e_{i_r}$, we have the following two cases: (i) there exists $1\leq k<j\leq r$ such that $1  \leq i_{k}=i_{j} \leq m$ and these boxes occupied by  $k$ and $j$ in $\Tc_0$ are in the same column and consecutive rows; (ii) there exists $1\leq k< r$ such that $m+1 \leq i_{k}=i_{k+1} \leq m+n$ and these boxes occupied by  $k,k+1$ in $\Tc_0$  are in the same row.

For case (i), recall the fusion procedure of primitive idempotents $\mathcal{E}_{\Tc}^{\lambda}$ \cite{Ju}:
\beq
h(\lambda)\mathcal{E}_{\Tc}^{\lambda}=\phi(z_1,\ldots,z_r)|_{z_1=c_1}\cdots |_{z_{r}=c_{r}},
\eeq
where $h(\lambda)$ is the hook length and denote
$$\phi(z_1,\ldots,z_r)=\mathop{\prod}\limits^{\rightarrow}_{1 \leq a<b \leq r}(1-\frac{P_{(a,b)}}{z_a-z_b}),$$
where the product is taken in the lexicographical order on the set of pairs $(a,b)$. By the fusion procedure and recurrence relations of primitive idempotents, we have that
\beq
\bal
&\Ec_{\Tc_0}^{\lambda}\cdot e_{i_1}\ot \cdots \ot e_{i_r}\\
&=Y_r\cdots Y_{j+1}\Ec_{\Tc^{-}}^{\mu} \cdot e_{i_1}\ot \cdots \ot e_{i_r}\\
&=Y_r\cdots Y_{j+1}\frac{1}{h(\mu)}\phi(z_1,\ldots,z_{j})|_{z_1=c_1}\cdots |_{z_{j}=c_{j}},
\eal
\eeq
where  $Y_k$ is a rational function in Jucys-Murphy element $y_k$ of $\Sym_r$, $j+1\leq k\leq r$,
and $\Tc^-$ is the standard tableau obtained from $\Tc_0$ by removing the boxes occupied by $j+1,\ldots,r$, denote by $\mu$ the shape of $\Tc^-$.

Then consider the last $j-k$ factors in $\phi(z_1,\ldots,z_{j})|_{z_1=c_1}\cdots |_{z_{j}=c_{j}}$, we have that
\beq
\bal
\mathop{\prod}\limits^{\rightarrow}_{k \leq t \leq j-1}(1-\frac{P_{(t,j)}}{c_t-c_j})\cdot e_{i_1}\ot \cdots \ot e_{i_r}&=const. (1-P_{(k,j)})\cdot e_{i_1}\ot \cdots \ot e_{i_r}\\
&=0.
\eal
\eeq

For case (ii), using the following relation from \cite[lemma 1.1.1]{Mo3}
\beq\label{right-act}
\mathcal{E}_{\Tc}^{\lambda}\cdot (P_{(a,a+1)}-\frac{1}{d_{a}(\Tc)})=\mathcal{E}_{\Tc}^{\lambda}P_{(a,a+1)}\mathcal{E}_{(a,a+1)\cdot\Tc}^{\lambda},
\eeq
and $\mathcal{E}_{(a,a+1)\cdot\Tc}^{\lambda}=0$ if the tableau $(a,a+1)\cdot \Tc$ is not standard. Then
\beq
\bal
&\mathcal{E}_{\Tc_0}^{\lambda}\cdot e_{i_1}\ot \cdots \ot e_{i_r}\\
=&-\mathcal{E}_{\Tc_0}^{\lambda}P_{(k,k+1)}\cdot e_{i_1}\ot \cdots \ot e_{i_r}\\
=&-\frac{1}{d_{k}(\Tc_0)}\mathcal{E}_{\Tc_0}^{\lambda}\cdot e_{i_1}\ot \cdots \ot e_{i_r}\\
=&-\mathcal{E}_{\Tc_0}^{\lambda}\cdot e_{i_1}\ot \cdots \ot e_{i_r}.
\eal
\eeq
Combining case  (i)  and  (ii), $\mathcal{E}_{\Tc_0}^{\lambda}\cdot e_{i_1}\ot \cdots \ot e_{i_r}=0$.

By the same argument, we can prove that for any $\Tc\in \mathrm{SYT}(\lambda)$,
\beq
\mathcal{E}_{\Tc}^{\lambda}\cdot e_{i_1}\ot \cdots \ot e_{i_r}=0.
\eeq
Then we show that equation \eqref{suff-con} by
\beq
\bal
&\mathcal{E}_{\Tc_0}^{\lambda}\cdot e_{i_{\tau(1)}}\ot \cdots \ot e_{i_{\tau(r)}}\\
=&(-1)^{\bar I_{\tau^{-1}}}\mathcal{E}_{\Tc_0}^{\lambda}P_{\tau^{-1}}\cdot e_{i_1}\ot \cdots \ot e_{i_r}.
\eal
\eeq
Using identity \eqref{right-act}, the above  equation can  express in terms of  $\mathcal{E}_{\Tc}^{\lambda}\cdot e_{i_1}\ot \cdots \ot e_{i_r}$, for some $\Tc\in \mathrm{SYT}(\lambda)$, so it is zero. Hence we obtain that $\Imm_{\chi^{\lambda}}(X_I)=0$, when $\lambda \notin H(m,n;r)$.
\end{proof}

\subsection{Super Schur-Weyl duality}
We review the double centralizer property  between the symmetric group and the general linear Lie superalgebra given by Berele and Regev \cite{BR1,BR2} and Sergeev \cite{Ser1,Ser2}. For more information, see \cite{CW}.

The general linear Lie superalgebra $\gl_{m|n}$ is spanned by the standard basis elements $E_{ij}, 1\leq i,j \leq m+n$. The $\mathrm{Z}_{2}$-grading of $\gl_{m|n}$ is defined by $E_{ij}\mapsto \bar{i}+\bar{j}$. The commutation relations are given by
\begin{equation}
[E_{ij},\ E_{kl}]=\delta_{kj}E_{il}-\delta_{il}E_{kj}(-1)^{(\bar{i}+\bar{j})(\bar{k}+\bar{l})}.
\end{equation}

We denote by $U(\gl_{m|n})$ the universal algebra of $\gl_{m|n}$.

The super vector space $(\CC^{m|n})^{\ot r}$ affords the following representation $\pi$ of $U(\gl_{m|n})$: for all $x\in \gl_{m|n}$,
\ben
x \cdot e_{i_1}\ot \cdots \ot e_{i_r}=\sum_{j=1}^r (-1)^{\bar x (\bar i_1+ \ldots + \bar i_{j-1})}e_{i_1}\ot \cdots \ot e_{i_{j-1}}\ot x e_{i_j} \ot e_{i_{j+1}}\ot \cdots \ot e_{i_r},
\een
which is the tensor representation of the fundamental one.

By the double centralizer property in \cite{BR1,BR2,Ser1,Ser2}, $(\CC^{m|n})^{\ot r}$ is a multiplicity free $U(\gl_{m|n})\times  \Sym_r $-module and there exists a decomposition of the space of tensors
\beq\label{Schur-Weyl-decom}
(\CC^{m|n})^{\ot r}\cong\bigoplus_{\lambda \in H(m,n;r)} U^{\lambda}\ot V^{\lambda},
\eeq
where $(\pi_{\lambda},U^{\lambda})$ is  the irreducible covariant tensor representation  for $U(\gl_{m|n})$ corresponding to  partition $\lambda$ and $(\rho_{\lambda},V^{\lambda})$ is the irreducible representation   of $\Sym_r$.
By a theorem of Kac \cite{Kac}, there must exist an integral dominant weight $\bm\lambda$ such that the highest weight module $L(\bm\lambda)$ is isomorphic to $U^{\lambda}$. Such a highest weight $\bm\lambda$ is called covariant highest weight.
The one to one correspondence between the covariant highest weights and partitions in $H(m,n)$ (cf. \eqref{e:partitionH}) is given as follows.
For any partition $\lambda$ in $H(m,n)$, the corresponding covariant highest weight $\bm\lambda=(\bm\lambda_1,\ldots, \bm\lambda_m \mid \bm\lambda_{m+1},\ldots, \bm\lambda_{m+n})$ is given by
\beq
\bal
\bm\lambda_{i}=\lambda_i, \ 1 \leq i \leq m,\\
\bm\lambda_{m+i}=\max\{0,\lambda'_i-m\}, \ 1\leq i \leq n,
\eal
\eeq
where $\lambda'$ is the partition conjugate to $\lambda$. Conversely if $\bm\lambda=(\bm\lambda_1,\ldots, \bm\lambda_m \mid \bm\lambda_{m+1},\ldots, \bm\lambda_{m+n})$ is a covariant highest weight
 then the components of $\lambda$ are given explicitly by
\beq
\bal
\lambda_i=\bm\lambda_i, \ 1\leq i \leq m,\\
\lambda_{m+i}= \sharp \{j \mid \bm\lambda_{m+j}\geq i, \ 1\leq j \leq n\}, \ 1\leq i \leq n.
\eal
\eeq

Let $\{v_{\Tc} \mid \Tc \in \mathrm{SYT}(\lambda)\}$ be the Young's orthonormal
basis of $V^{\lambda}$.
Define a nondegenerate symmetric  bilinear form $\langle \cdot | \cdot \rangle$ on  $(\CC^{m|n})^{\ot r}$ by
\beq\label{inner-product}
\bal
\langle e_{i_1}\ot \cdots \ot e_{i_r}, e_{j_1}\ot \cdots \ot e_{j_r}  \rangle&=\prod_{1 \leq k \leq r}\delta_{i_k,j_k}.
\eal
\eeq
We introduce the $*$-operations for $\Sym_r $  and $U(\gl_{m|n})$  respectively. Let  $\si^{*}=\si^{-1} \in \Sym_r$. And $E_{ij}^{*}=E_{ji}, 1 \leq i,j \leq m+n$ in $U(\gl_{m|n})$.  They are involutive anti-automorphisms.
\begin{lemma}\label{adjoint-op}
The $*$-operations for $\Sym_r$  and $U(\gl_{m|n})$  afford a contravariant  bilinear form, i.e.
\beq\label{contravariant-act}
\bal
  &\langle x\cdot  e_{i_1}\ot \cdots \ot e_{i_r}, e_{j_1}\ot \cdots \ot e_{j_r}  \rangle\\
  &=\langle  e_{i_1}\ot \cdots \ot e_{i_r}, x^* \cdot e_{j_1}\ot \cdots \ot e_{j_r}  \rangle,
\eal
\eeq
where $x$ is any element in $\Sym_r$ or $U(\gl_{m|n})$.
\end{lemma}
\begin{proof}
We denote $I=(i_1, \ldots , i_r)$ and $J=(j_1, \ldots , j_r)$.

For any $\si \in \Sym_r$ we have that
\ben
\bal
&\langle P_{\si}\cdot e_{i_1}\ot \cdots \ot e_{i_r}, e_{j_1}\ot \cdots \ot e_{j_r}  \rangle\\
&=(-1)^{\bar I_{\si}}\langle e_{i_{\si^{-1}(1)}}\ot \cdots \ot e_{i_{\si^{-1}(r)}}, e_{j_1}\ot \cdots \ot e_{j_r}  \rangle\\
&=(-1)^{\bar I_{\si}}\prod_{1 \leq k \leq r}\delta_{i_{\si^{-1}(k)},j_k},
\eal
\een
and
\ben
\bal
&\langle e_{i_1}\ot \cdots \ot e_{i_r}, P_{\si^{-1}}\cdot e_{j_1}\ot \cdots \ot e_{j_r}  \rangle\\
&=(-1)^{\bar J_{\si^{-1}}}\langle e_{i_1}\ot \cdots \ot e_{i_r}, e_{j_{\si(1)}}\ot \cdots \ot e_{j_{\si(r)}}  \rangle\\
&=(-1)^{\bar J_{\si^{-1}}}\prod_{1 \leq k \leq r}\delta_{i_{k},j_{\si(k)}}.
\eal
\een
Since  when $J=\si(I)$,
\ben
\bal
(-1)^{\bar I_{\si}}&=\prod_{k<t, \atop \si(k)> \si(t)}(-1)^{\bar i_k \bar i_t}=\prod_{k> t, \atop\si^{-1}(k)<\si^{-1}(t)}(-1)^{\bar i_{\si^{-1}(k)} \bar i_{\si^{-1}(t)}}\\
&=\prod_{k>t, \atop \si^{-1}(k) <\si^{-1}(t)}(-1)^{\bar j_k \bar j_t}=(-1)^{\bar J_{\si^{-1}}},
\eal
\een
so the bilinear form  on $(\CC^{m|n})^{\ot r}$ is contravariant  for $\Sym_r$.

For $x=E_{ji} \in \gl_{m|n}$, if $i \notin I$ or $j \notin J$, the left side of \eqref{contravariant-act} is zero. And the action of $E_{ji}^*$ on right side of \eqref{contravariant-act} is also trivial. Unless, for any  $1 \leq s,t \leq r$,
\ben
\bal
&\langle E_{j_s,i_t}\cdot e_{i_1}\ot \cdots \ot e_{i_r}, e_{j_1}\ot \cdots \ot e_{j_r}  \rangle\\
&=\sum_{k=1}^r(-1)^{(\bar j_s+\bar i_t) (\bar i_1+ \ldots + \bar i_{k-1})}\langle e_{i_1}\ot \cdots \ot e_{i_{k-1}}\ot E_{j_s,i_t}\cdot e_{i_k} \ot \cdots \ot e_{i_r}, e_{j_1}\ot \cdots \ot e_{j_r}  \rangle\\
&=\sum_{k=1}^r(-1)^{(\bar j_s+\bar i_t) (\bar i_1+ \ldots + \bar i_{k-1})} \delta_{i_{t},i_{k}}\delta_{j_{s},j_{k}}\prod_{p \neq k}\delta_{i_{p},j_{p}}.
\eal
\een
And
\ben
\bal
&\langle e_{i_1}\ot \cdots \ot e_{i_r}, (-1)^{\bar i_t+\bar j_s} E_{i_t,j_s}\cdot e_{j_1}\ot \cdots \ot e_{j_r}  \rangle\\
&=\sum_{k=1}^r(-1)^{(\bar j_s+\bar i_t) (\bar j_1+ \ldots + \bar j_{k-1})}\langle e_{i_1}\ot \cdots \ot e_{i_r}, e_{j_1}\ot \cdots \ot e_{j_{k-1}}\ot E_{i_t,j_s}\cdot e_{j_k} \ot \cdots \ot e_{j_r}  \rangle\\
&=\sum_{k=1}^r(-1)^{(\bar j_s+\bar i_t) (\bar j_1+ \ldots + \bar j_{k-1})} \delta_{j_{s},j_{k}}\delta_{i_{k},i_{t}}\prod_{p \neq k}\delta_{i_{p},j_{p}}.
\eal
\een
Hence the bilinear form on $(\CC^{m|n})^{\ot r}$ is contravariant  for  $U(\gl_{m|n})$.
\end{proof}
\subsection{Gelfand-Tsetlin type basis for covariant representations}
A supertableau $\Lambda$ of shape $\lambda \in H(m,n;r)$ is semistandard  if it satisfies the following conditions:

(a) the entries weakly increase from left to right along each row and down each column;

(b) the entries in $\{1,\ldots, m\}$ strictly increase down each column;

(c) the entries in $\{m + 1,\ldots, m+n\}$ strictly increase from left to right along each row.

We denote by $SSYT(\lambda)$ the set of semistandard supertableaux of shape $\lambda$. We say a supertableau $\Lambda$ is of weight $\mu$ if  $\mu_i$ equals the number of $i's$ in $\Lambda$.

Molev \cite{Mo2} constructed a basis $\xi_{\Lambda}$ parameterized by all semistandard supertableaux $\Lambda$ of shape $\lambda$ for the irreducible covariant tensor representation $U^{\lambda}\cong L(\bm\lambda)$ and gave the explicit action of the generators of $U(\gl_{m|n})$ in the basis.
On the other hand, Stoilova and van der Jeugt \cite{SJ} constructed Gelfand-Tsetlin type basis and provided different matrix element formulas of the generators of $U(\gl_{m|n})$ in the basis. Explicitly, they showed that
the set of vectors $\zeta_{\Lambda}=(\lambda_{ij})_{1 \leq j \leq i \leq m+n}$
satisfying the conditions
\begin{equation}
 \begin{array}{rl}
(1) & \lambda_{m+n,j}=\bm \lambda_j , \ 1\leq j\leq m+n;\\
(2)&\lambda_{pi}-\lambda_{p-1,i}\equiv\theta_{p-1,i}\in\{0,1\},\ 1\leq i\leq m, m+1\leq p\leq m+n;\\
(3)  &  \lambda_{pm}\geq \# \{i:\lambda_{pi}>0, m+1\leq i \leq p\}, \  m+1\leq p\leq m+n ;\\
(4)& \text{if }
\lambda_{m+1,m}=0, \text{then}\; \theta_{mm}=0; \\
(5)& \lambda_{pi}-\lambda_{p,i+1}\in{\mathbb Z}_+,\ 1\leq i\leq m-1,
    m+1\leq p\leq r-1;\\
(6)& \lambda_{i,j}-\lambda_{i-1,j}\in{\mathbb Z}_+\text{ and }\lambda_{i-1,j}-\mu_{i,j+1}\in{\mathbb Z}_+,\\
  & 1\leq j\leq i\leq m\text{ or } m+1\leq j\leq i\leq m+n.
 \end{array}
\end{equation}
constitutes a basis in $L(\bm\lambda)$.

In  \cite{Mo2}, Molev shows that each semistandard supertableau  is one-to-one corresponding to the pair of patterns which satisfying betweenness (or interlacing) conditions.

Moreover, these Gelfand-Tsetlin type basis $\zeta_{\Lambda}$ defined by Stoilova and van der Jeugt are also one-to-one corresponding to semistandard supertableaux $\Lambda$ such that for any $1 \leq i \leq m+n$,\\
if $1 \leq j\leq m$,
\beq\label{pattern-tableau-corr}
\bal
\lambda_{ij}& = \text{the number of entries }\leq i\text{ in the }j\text{th row of }\Lambda;
\eal
\eeq
if $1 \leq k\leq i-m \leq n$,
$\lambda_{i,m+k}=$ the number of entries $\leq i$  in the $k$th column of $\Lambda$ starting from $(m+1)$-th row to the last row.

Let $\hat{E}$  be the $(m+n)\times(m+n)$  matrix with coefficients in $U(\gl_{m|n})$ whose $ij$-th entry equals $\hat{E}_{ij}= (-1)^{\bar{j}}E_{ij}$.
The {\it  Berezinian} is an element in $U(\gl_{m|n})[[t]]$ defined by
\begin{equation}
\begin{split}
B_{m+n}(t)=\sum_{\si\in \Sym_{m}} sgn  (\si)(1+t\hat{E})_{\si(1),1}\cdots(1+t(\hat{E}-m+1))_{\si(m),m}
\\
\times\sum_{\tau\in \Sym_{n}}
 sgn(\tau)(1+t(\hat{E}-m+1))_{m+1,m+\tau(1)}^{-1}\cdots(1+t(\hat{E}-m+n))_{m+n,m+\tau(n)}^{-1}.
\end{split}
\end{equation}
Nazarov showed that the coefficients of the Berezinian belong to the center of $U(\gl_{m|n})$ \cite{Na91}. In fact, they generate the center of $U(\gl_{m|n})$ according to \cite{Gow}.

Let  $l_{i}=\bm\lambda_{i}-i+1$ for $i=1, \ldots, m,$
$l_{j}=-\bm\lambda_{j}+j-2m$ for $j=m+1, \ldots, m+n.$
Then $B_{m+n}(t)$  acts on $L(\bm \lambda)$ by the scalar \cite{Mo1,MRe}
\begin{equation}
\frac{(1+tl_{1})\cdots(1+tl_{m})}
{(1+tl_{m+1})\cdots(1+tl_{m+n})}.
\end{equation}
Denote $l_{ij}=\lambda_{ij}-j+1$ for $j=1, \ldots, m,$
$l_{ij}=-\lambda_{ij}+j-2m$ for $j=m+1, \ldots, m+n$. From the branching rules of $L(\bm\lambda)$, we can obtain that
\beq\label{central-char}
\bal
 &B_k(t)\zeta_{\Lambda}=\left(1+t l_{k1}\right)\cdots\left(1+t l_{kk}\right) \zeta_{\Lambda} \ \ \text{for} \ 1 \leq k \leq m,\\
 &B_k(t)\zeta_{\Lambda}=\frac{(1+tl_{k1})\cdots(1+tl_{km})}
{(1+tl_{k,m+1})\cdots(1+tl_{kk})} \zeta_{\Lambda} \ \ \text{for} \ m+1 \leq k \leq m+n.
\eal
\eeq

\subsection{Weight space interpretation of super-immanants}
We can decompose \eqref{Schur-Weyl-decom} into basis vectors:
\beq\label{Schur-Weyl-decom2}
(\CC^{m|n})^{\ot r}\cong\sum_{\lambda \in H(m,n;r)} \sum_{(\Lambda,\Tc)} \mathbb C \zeta_{\Lambda}\ot v_{\Tc}.
\eeq
 By Lemma \ref{adjoint-op}, these basis vectors in decomposition \eqref{Schur-Weyl-decom2} are normalized orthogonal, i.e.
$$\langle \zeta_{\Lambda}\ot v_{\Tc}, \zeta_{\Lambda'}\ot v_{\Tc'}\rangle =\delta_{\Lambda\Lambda'}\delta_{\Tc\Tc'}$$
for any semistandard supertableaux $\Lambda,\Lambda'$ and standard  Young tableaux $\Tc,\Tc'$.

An $(m+n)$-tuple $(a_1,\ldots,a_{m+n})$ in $\mathbb{Z}_{\geq 0}^{m+n}$ such that
$a_1+\cdots +a_{m+n}=r$ is called  a weak composition of $r$ into $(m+n)$ parts or a weak $(m+n)$-composition of $r$.
It's clear that
the weights of $U(\gl_{m|n})$ in $(\CC^{m|n})^{\ot r}$ are exactly weak $(m+n)$-compositions of $r$.
For any weak $(m+n)$-composition $\bm\mu$ of $r$, define the projection operator
$$\mathcal P_{\bm\mu}: (\CC^{m|n})^{\ot r}\rightarrow ((\CC^{m|n})^{\ot r})_{\bm \mu},$$
then $\mathcal P_{\bm \mu}: L(\bm\lambda)\rightarrow L(\bm\lambda)_{\bm \mu}$, where $( (\CC^{m|n})^{\ot r})_{\bm \mu}$ and $L(\bm\lambda)_{\bm \mu}$ are the subspaces with weight $\bm \mu$ as  $U(\gl_{m|n})$-modules.

By the bialgebra duality between $U(\gl_{m|n})$ and $A(\Mat_{m|n})$, the super tensor space $(\CC^{m|n})^{\ot r}$ has an $A(\Mat_{m|n})$-comodule structure given by
\beq
\bal
  \Delta: &  (\CC^{m|n})^{\ot r} \rightarrow A(\Mat_{m|n})\ot  (\CC^{m|n})^{\ot r}\\
  & e_{j_1}\ot \cdots \ot e_{j_r}\mapsto \sum_{(i_1,\ldots,i_r)}(-1)^{\sum\limits_{a<b}\bar i_a(\bar i_b+\bar j_b)} x_{i_1,j_1}\cdots x_{i_r,j_r} \ot e_{i_1}\ot \cdots \ot e_{i_r},
\eal
\eeq
where the sum is over all sequence $(i_1\ldots,i_r)$ of $[m+n]$.

Now we give a super-analog of general Kostant theorem \cite{Ko}, providing a weight space interpretation of super-immanants.
It is convenient to write $I=(1^{\alpha_1},\ldots,(m+n)^{\alpha_{m+n}})$ to specify the multiplicity $\alpha_i$ of $i$ in the non-decreasing multiset $I=(i_1\leq \ldots \leq i_r)$. Here $\alpha_i=\alpha_i(I)= \mathrm{Card}\{j\in I \mid j=i\}$.

Let  $$\alpha({I})=\alpha_1!\alpha_2 ! \cdots \alpha_{m+n} !$$
Denote by $$\mathfrak S_{I}=\mathfrak S_{\alpha_1}\times\mathfrak S_{\alpha_2} \times \cdots \times \mathfrak S_{\alpha_{m+n}}$$  the Young subgroup  of $\mathfrak{S}_r$, and $\mathfrak S_{\alpha_j}=1$ if $\alpha_j=0$.
\begin{theorem}\label{super-Kostant-thm}
  Let $\lambda \in H(m,n;r)$, $\bm\mu$ be  a weak  $(m+n)$-composition of $r$, and assume the multiset $I=(1^{\mu_1},2^{\mu_2},\ldots,(m+n)^{\mu_{m+n}})$ of $[m+n]$. Then
\beq\label{super-Kostant-iden}
\frac{\Imm_{\chi^{\lambda}}(X_I)}{\alpha(I)}=str(\mathcal P_{\bm\mu}\ot 1)\circ\Delta\circ \mathcal P_{\bm \mu}|_{L(\bm\lambda)}.
\eeq
\end{theorem}
Before proving this theorem, we give
an explicit correspondence between the Gelfand-Tsetlin type bases of the  covariant representations of $U(\gl_{m|n})$ and
Young's orthonormal basis of  irreducible representations of $\Sym_r$.

 Given any weight $\bm\mu=(\bm\mu_1,\ldots,\bm\mu_{m+n})$, which can be written as a multiset $I=(1^{\mu_1} ,2^{\mu_2} ,\ldots ,(m+n)^{\mu_{m+n}})=(i_1, \ldots, i_r)$. We define  the following map
\beq
\bal
\theta_{\bm\mu}: SYT(\lambda) \rightarrow YT(\lambda)\\
\eal
\eeq
which maps any standard Young tableau $\Tc$ to a Young tableau $\theta_{\bm\mu}(\Tc)$ that replaces each node $k$ in $\Tc$  by $i_k$ for $1 \leq k \leq r$.
Note that $SSYT(\lambda)$ is contained in $YT(\lambda)$.
\begin{lemma}\label{schur-weyl-corr}
Let $\lambda \vdash r$ and $\bm\mu$ be a weak $(m+n)$-composition of $r$ written as the multiset  $I=(1^{\mu_1},2^{\mu_2},\ldots,(m+n)^{\mu_{m+n}})=(i_1 \leq \ldots \leq i_r)$ of $[m+n]$, let  $\mathcal{E}_{\Tc}^{\lambda}$ be the primitive idempotent   of $\Sym_r$ associated with ${\Tc}$. Then as a $ U(\mathfrak{gl}_{m|n}) \times \Sym_r $-module
\begin{align}\label{schur-weyl-iden1}
\sqrt{\frac{h(\lambda)}{\alpha(I)}}\mathcal{E}_{\Tc}^{\lambda}\cdot e_{i_1}\ot \cdots \ot e_{i_r}=\left\{\begin{array}{cc}
c\cdot \zeta_{\theta_{\bm\mu}(\Tc)}\ot v_{\Tc}&\text{ if } \theta_{\bm\mu}(\Tc)  \in SSYT(\lambda),\\
0&\text{ if }\theta_{\bm\mu}(\Tc)\notin  SSYT(\lambda),
\end{array}\right.
\end{align}
where $h(\lambda)$  is the hook length and $c$ is a nonzero constant. Explicitly,
$c=1$, if $\Tc$ is the unique pre-image of $\theta_{\bm\mu}(\Tc)$. And
$\sum_{i=1}^s||c_i||^2=1$ for $\theta_{\bm\mu}(\Tc_1)=\cdots=\theta_{\bm\mu}(\Tc_s)$ $(s>1)$.
In particular,
assume $r \leq m+n$, then
\beq\label{schur-weyl-iden2}
\sqrt{h_{\lambda}} \mathcal{E}_{\Tc}^{\lambda}\cdot e_1\ot \cdots \ot e_r=\zeta_{\Tc}\ot v_{\Tc} .
\eeq
\end{lemma}
\begin{proof}
(i) We first prove identity \eqref{schur-weyl-iden2}. It is clear that $\lambda \in H(m,n;r)$ when $r \leq m+n$.
By the super Schur-Weyl duality,  $\mathcal{E}_{\Tc}^{\lambda}$ is the projection operator from $(\CC^{m|n})^{\ot r} $ to $L(\bm\lambda)$. Since
\beq
\bal
 &\langle  \mathcal{E}_{\Tc}^{\lambda}\cdot e_1\ot \cdots \ot e_r, \mathcal{E}_{\Tc'}^{\lambda}\cdot e_1\ot \cdots \ot e_r\rangle\\
 &=\langle  \mathcal{E}_{\Tc'}^{\lambda}\mathcal{E}_{\Tc}^{\lambda}\cdot e_1\ot \cdots \ot e_r, e_1\ot \cdots \ot e_r\rangle\\
 &=\delta_{\Tc\Tc'}\langle  \mathcal{E}_{\Tc}^{\lambda}\cdot e_1\ot \cdots \ot e_r, e_1\ot \cdots \ot e_r\rangle\\
 &=\frac{1}{h(\lambda)}\delta_{\Tc\Tc'}.
\eal
\eeq
So
$\sqrt{h(\lambda)} \mathcal{E}_{\Tc}^{\lambda}\cdot e_1\ot \cdots \ot e_r\subset L(\bm\lambda)\ot v_{\Tc}$.

Moreover by identities \eqref{central-char}, $B_{k}(t), \ r\leq  k \leq m+n$ acts on  $\sqrt{h(\lambda)} \mathcal{E}_{\Tc}^{\lambda}\cdot e_1\ot \cdots \ot e_r$  as a scalar multiplication and only depends on highest weight $\bm\lambda$.
Consider the action of $B_{k}(t), 1 \leq k <r$, from the recurrence relation of $\mathcal{E}_{\Tc}^{\lambda}$ we have that
\beq
\bal
&B_{k}(t)\cdot \mathcal{E}_{\Tc}^{\lambda} e_1\ot \cdots \ot e_r\\
&=Y_r\cdots Y_{k+1} B_k(t)\cdot \mathcal{E}_{\Tc^-}^{\nu} e_1\ot \cdots \ot e_r,
\eal
\eeq
where $Y_{j}$ is a rational function in  Jucys-Murphy element $y_j$ of $\Sym_r$, $k+1 \leq j \leq r$, and $\Tc^-$ is the standard tableau obtained from $\Tc$ by removing the boxes occupied by $k+1,\ldots,r$, denote by $\nu$ the shape of $\Tc^-$.

The action of $ B_k(t)\cdot \mathcal{E}_{\Tc^-}^{\nu} e_1\ot \cdots \ot e_r$ is scalar and only depends on standard tableau $\Tc^-\in SYT(\nu) \subset SSYT(\nu)$ by identities \eqref{central-char} and correspondence \eqref{pattern-tableau-corr}. Hence
\beq
\sqrt{h(\lambda)} \mathcal{E}_{\Tc}^{\lambda}\cdot e_1\ot \cdots \ot e_r= \zeta_{\Tc}\ot v_{\Tc} .
\eeq

(ii) To prove \eqref{schur-weyl-iden1}, we first note that the argument of proposition \ref{imm-char} implies that if $\lambda \notin H(m,n;r)$,
\beq
\mathcal{E}_{\Tc}^{\lambda}\cdot e_{i_1}\ot \cdots \ot e_{i_r}=0.
\eeq
 Consider $\lambda \in H(m,n;r)$ below. If $\theta_{\bm \mu}(\Tc) \notin \mathrm{SSYT}(\lambda)$, then there exists  either one or both of the two cases:
\beq
\bal
&\text{(i)}\ p,q \ \text{are in the same column of}\ \Tc \ \text{and}\ i_p=i_q \ \text{with}\ \bar i_p=0;\\
&\text{(ii)}\ s,t \ \text{are in the same row of}\ \Tc \ \text{and} \ i_s=i_t \ \text{with} \ \bar i_s=1.
\eal
\eeq
We only prove the second case. It is similar to the proof of  proposition \ref{imm-char}.
 We can assume that $1 \leq s<t \leq r$ and content $c_s(\Tc)-c_t(\Tc)=-1$. So $i_s=i_{s+1}=\ldots=i_t$. Moreover, by the  recurrence relation and fusion procedure of $\Ec_{\Tc}^{\lambda}$, we have that
\beq
\bal
&\Ec_{\Tc}^{\lambda}\cdot e_{i_1}\ot \cdots \ot e_{i_r}\\
&=Y_r\cdots Y_{t+1}\Ec_{\Tc^{-}}^{\mu} \cdot e_{i_1}\ot \cdots \ot e_{i_r}\\
&=Y_r\cdots Y_{t+1}\frac{1}{h(\mu)}\phi(z_1,\ldots,z_t)|_{z_1=c_1}\cdots |_{z_{t}=c_{t}}\cdot e_{i_1}\ot \cdots \ot e_{i_r}.
\eal
\eeq
 Since
$$\phi(z_1,\ldots,z_t)=\mathop{\prod}\limits^{\rightarrow}_{1 \leq a<b \leq t}(1-\frac{P_{(a,b)}}{z_a-z_b}),$$
we can rewrite the fusion procedure as
\beq
\bal
\Ec_{\Tc^{-}}^{\mu} =&\frac{1}{h(\mu)} \phi(z_1,\ldots,z_{t-1})\mathop{\prod}\limits^{\rightarrow}_{1 \leq a \leq t-1}(1-\frac{P_{(a,t)}}{z_a-z_t})|_{z_1=c_1}\cdots |_{z_{t}=c_{t}}\\
=&\frac{1}{h(\mu)} \phi(z_1,\ldots,z_{t-1})\mathop{\prod}\limits^{\rightarrow}_{1 \leq a \leq s-1}(1-\frac{P_{(a,t)}}{z_a-z_t})|_{z_1=c_1}\cdots |_{z_{t}=c_{t}}\\
&\times \mathop{\prod}\limits^{\rightarrow}_{s \leq b \leq t-1}(1-\frac{P_{(b,t)}}{c_b-c_t}),
\eal
\eeq
where $\phi(z_1,\ldots,z_{t-1})$ is the front part of $\phi(z_1,\ldots,z_{t})$, with subscripts ranging from $1$ to $t-1$. Note that
\ben
(1-\frac{P_{(s,t)}}{c_s-c_t})=1+P_{(s,t)},
\een
and
\ben
\bal
&\mathop{\prod}\limits^{\rightarrow}_{s \leq b \leq t-1}(1-\frac{P_{(b,t)}}{c_b-c_t}) \cdot e_{i_1}\ot \cdots \ot e_{i_r}\\
&=const\cdot  (1+P_{(s,t)})\cdot e_{i_1}\ot \cdots \ot e_{i_r}\\
&=0.
\eal
\een
Hence, $\Ec_{\Tc}^{\lambda}\cdot e_{i_1}\ot \cdots \ot e_{i_r}=0$ if $\theta_{\bm\mu}(\Tc) \notin \mathrm{SSYT}(\lambda)$.

Finally, we consider $\theta_{\bm \mu}(\Tc) \in \mathrm{SSYT}(\lambda)$. We have known from identity \eqref{schur-weyl-iden2} that for $r\leq m+n$  and $\bm\mu=(1^{r})$,
\beq
\sqrt{h_{\lambda}} \mathcal{E}_{\Tc}^{\lambda}\cdot e_1\ot \cdots \ot e_r=\zeta_{\Tc}\ot v_{\Tc} .
\eeq

It is sufficient to show  equation \eqref{schur-weyl-iden1} when $r \leq m+n$. Indeed for $r>m+n$, as $U(\mathfrak{gl}_{m|n})$ is a subalgebra of $U(\mathfrak{gl}_{m|r-m})$, then the actions of $B_1(t),\ldots,B_{m+n}(t)$ unchanged in restricting  $U(\mathfrak{gl}_{m|r-m})$ to $U(\mathfrak{gl}_{m|n})$.

 For $r\leq m+n$ and  $I=(i_1 \leq \ldots \leq i_r)$,   we denote $E_i=E_{i,i+1},F_i=E_{i+1,i}$ in $U(\gl_{m|n})$.
Note that
\beq
\bal
&\left(\overset{\rightarrow}{\prod}_{i_s>s}F_{i_s-1}\cdots F_s \right) \left(\overset{\leftarrow}{\prod}_{i_t<t}E_{i_t}\cdots E_{t-1}\right)\cdot e_1\ot \cdots \ot e_r\\
&=c\cdot e_{i_1}\ot \cdots \ot e_{i_r},
\eal
\eeq
where $c$ is some nonzero constant.

Moreover assume that $t_0$ is the minimal such that $i_{t_0}<t_0$, then $i_{t_0-1}\geq t_0-1$. Since
$$
\begin{matrix}
\cdots & t_0-1 & < & t_0 & \cdots &  \\
 & \rotatebox{90}{$\geq$} & & \rotatebox{90}{$<$} &  & \\
\cdots  &i_{t_0-1} &\leq& i_{t_0}& \cdots &
\end{matrix}$$
so $i_{t_0-1}= i_{t_0}=t_0-1$.
According to coresponding \eqref{pattern-tableau-corr} and Gelfand-Tsetlin type formula \cite{SJ}, $E_{t_0-1}\zeta_{\Tc}=0$ if $\overline{t_0-1}=0$ and $t_0-1,t_0$ are in the same column of $\Tc$ or $\overline{t_0-1}=1$ and $t_0-1,t_0$ are in the same row of $\Tc$. Otherwise,
\beq
\bal
E_{t_0-1}\zeta_{\Tc}=c'\zeta_{\theta_{\bm\mu_0}(\Tc)},
\eal
\eeq
where $c'$ is some nonzero constant, $\bm \mu_0=(1^{t_0-2},2,0,1^{r-t_0},0^{m+n-r})$ and $\theta_{\bm\mu_0}(\Tc)$ is the unique semistandard supertableau of shape $\lambda$ and weight $\bm \mu_0$ by replacing $t_0$ with $t_0-1$ in $\Tc$. Then
\ben
\bal
&\sqrt{h(\lambda)} E_{t_0-1}\cdot \mathcal{E}_{\Tc}^{\lambda} e_1\ot \cdots \ot e_r\\
&=\sqrt{h(\lambda)}\mathcal{E}_{\Tc}^{\lambda} e_1\ot \cdots \ot e_{t_0-1}\ot e_{t_0-1} \ot e_{t_0+1}\ot \cdots \ot e_r.
\eal
\een
On the other hand,
\ben
\bal
&E_{t_0-1}\cdot \zeta_{\Tc}\ot v_{\Tc}=const\cdot \zeta_{\theta_{\bm \mu_0}(\Tc)}\ot v_{\Tc},
\eal
\een
here these constants are nonzero.
Therefore,
\ben
\bal
&\sqrt{h(\lambda)} \mathcal{E}_{\Tc}^{\lambda} e_1\ot \cdots \ot e_{t_0-1}\ot e_{t_0-1} \ot e_{t_0+1}\ot \cdots \ot e_r=const\cdot \zeta_{\theta_{\bm \mu_0}(\Tc)}\ot v_{\Tc}.
\eal
\een
Generally, if $i_{t}<t$, consider the order of actions $E_{i_t}\cdots E_{t-1}$ and these actions are not trivial, we have that the unique Gelfand-Tsetlin type vector $\zeta_{\Lambda}$ and $\Lambda$ is obtained from the semistandard supertableau in previous step by replacing  $t$ with $i_t$.

As the two actions of $\overset{\rightarrow}{\prod}_{i_s>s}F_{i_s-1}\cdots F_s $ and  $\overset{\leftarrow}{\prod}_{i_t<t}E_{i_t}\cdots E_{t-1}$ are supercommutative, we can show the similar results when $i_s>s$. Hence
\beq
\bal
&\left(\overset{\rightarrow}{\prod}_{i_s>s}F_{i_s-1}\cdots F_s \right) \left(\overset{\leftarrow}{\prod}_{i_t<t}E_{i_t}\cdots E_{t-1}\right)  \mathcal{E}_{\Tc}^{\lambda} e_1\ot \cdots \ot e_r\\
&=const\cdot \mathcal{E}_{\Tc}^{\lambda} e_{i_1}\ot \cdots \ot e_{i_r}.
\eal
\eeq
\beq
\bal
&\left(\overset{\rightarrow}{\prod}_{i_s>s}F_{i_s-1}\cdots F_s \right) \left(\overset{\leftarrow}{\prod}_{i_t<t}E_{i_t}\cdots E_{t-1}\right)  \mathcal{E}_{\Tc}^{\lambda}\cdot\zeta_{\Tc}\ot v_{\Tc}\\
&=const \cdot \zeta_{\theta_{\bm\mu}(\Tc)}\ot v_{\Tc}.
\eal
\eeq
Both constants are nonzero. Hence
\beq
\bal
\sqrt{\frac{h(\lambda)}{\alpha(I)}}\mathcal{E}_{\Tc}^{\lambda}\cdot e_{i_1}\ot \cdots \ot e_{i_r}=
c\cdot \zeta_{\theta_{\bm\mu}(\Tc)}\ot v_{\Tc}.
\eal
\eeq
Explicitly, we consider the  bilinear form between them.
If $\Tc$ is the unique pre-image of $\theta_{\bm\mu}(\Tc)$, ( i.e. if $i_s=i_t$ with $\bar i_s=0$ in $\theta_{\bm\mu}(\Tc)$, then $s,t$ are in the same row of $\Tc$ or with $\bar i_s=1$ in $\theta_{\bm\mu}(\Tc)$, then $s,t$ are in the same column of $\Tc$), then
\beq
\bal
&\langle  \mathcal{E}_{\Tc}\cdot e_{i_1}\ot \cdots \ot e_{i_r}, \mathcal{E}_{\Tc'}\cdot e_{i_1}\ot \cdots \ot e_{i_r}\rangle\\
&=\langle  \mathcal{E}_{\Tc'}\mathcal{E}_{\Tc}\cdot e_{i_1}\ot \cdots \ot e_{i_r},e_{i_1}\ot \cdots \ot e_{i_r}\rangle\\
&=\delta_{\Tc\Tc'}\langle  \mathcal{E}_{\Tc}\cdot e_{i_1}\ot \cdots \ot e_{i_r}, e_{i_1}\ot \cdots \ot e_{i_r} \rangle \\
&=\delta_{\Tc\Tc'}\frac{1}{h(\lambda)}
\langle  \sum_{\si \in \Sym_r}\langle \si^{-1}v_{\Tc},v_{\Tc}\rangle P_{\si}\cdot e_{i_1}\ot \cdots \ot e_{i_r}, e_{i_1}\ot \cdots \ot e_{i_r} \rangle \\
&=\delta_{\Tc\Tc'}\frac{1}{h(\lambda)}
\langle  \sum_{\si \in \Sym_I}\langle \si^{-1}v_{\Tc},v_{\Tc}\rangle P_{\si}\cdot e_{i_1}\ot \cdots \ot e_{i_r}, e_{i_1}\ot \cdots \ot e_{i_r} \rangle.
\eal
\eeq
Hence,
\beq
\bal
\langle  \mathcal{E}_{\Tc}\cdot e_{i_1}\ot \cdots \ot e_{i_r}, \mathcal{E}_{\Tc'}\cdot e_{i_1}\ot \cdots \ot e_{i_r}\rangle=\delta_{\Tc\Tc'}\frac{\alpha(I)}{h(\lambda)}.
\eal
\eeq
As these basis vectors in decomposition \eqref{Schur-Weyl-decom2} are normalized orthogonal,
\beq
\sqrt{\frac{h(\lambda)}{\alpha(I)}}\mathcal{E}_{\Tc}\cdot e_{i_1}\ot \cdots \ot e_{i_r}= \zeta_{\theta_{\bm\mu}(\Tc)}\ot v_{\Tc}.
\eeq

Moreover if there exists $\theta_{\bm\mu}(\Tc_1)=\cdots=\theta_{\bm\mu}(\Tc_s)$ ($s>1$), where $\Tc_1,\ldots,\Tc_s$ are standard Young tableaux with same shape $\lambda$, then
\begin{align*}
&\langle  \sqrt{\frac{h(\lambda)}{\alpha(I)}}\sum_{k=1}^s\mathcal{E}_{\Tc_k}\cdot e_{i_1}\ot \cdots \ot e_{i_r}, \sqrt{\frac{h(\lambda)}{\alpha(I)}}\sum_{k=1}^s\mathcal{E}_{\Tc_k}\cdot e_{i_1}\ot \cdots \ot e_{i_r}\rangle\\
&=\frac{h(\lambda)}{\alpha(I)}\langle  \sum_{k=1}^s\mathcal{E}_{\Tc_k}\cdot e_{i_1}\ot \cdots \ot e_{i_r},e_{i_1}\ot \cdots \ot e_{i_r}\rangle\\
&=\frac{1}{\alpha(I)}\langle \sum_{k=1}^s\sum_{\si \in \Sym_r}\langle \si^{-1} v_{\Tc_k},v_{\Tc_k}\rangle P_{\si}\cdot e_{i_1}\ot \cdots \ot e_{i_r}, e_{i_1}\ot \cdots \ot e_{i_r} \rangle \\
&=\frac{1}{\alpha(I)}\sum_{\si \in \Sym_I}\sum_{k=1}^s\langle \si^{-1} v_{\Tc_k},v_{\Tc_k}\rangle \langle  P_{\si}\cdot e_{i_1}\ot \cdots \ot e_{i_r}, e_{i_1}\ot \cdots \ot e_{i_r} \rangle.
\end{align*}
Since the subspace $V=\text{span}_{\mathbb{C}}\{v_{\Tc_k}, 1 \leq k \leq s\}$ is a finite-dimensional module of $\Sym_{\bm\mu}=\Sym_{\bm\mu_1}\times \cdots \times \Sym_{\bm\mu_r}$ and it is isomorphic to the tensor product $V_1\ot \cdots \ot V_r$ of skew representation of $\Sym_{\bm\mu_i}$, $ 1 \leq i \leq r$. As $\theta_{\bm\mu}(\Tc_k)\neq 0$, the skew partition is disconnected and then it is a permutation module. Thus it contains the trivial representation $1_{\Sym_{\mu}}$  with multiplicity one. By the definition of the inner product of character,
\beq
\bal
&\frac{1}{\alpha(I)}\sum_{\si \in \Sym_I}\sum_{k=1}^s\langle \si^{-1} v_{\Tc_k},v_{\Tc_k}\rangle \langle  P_{\si}\cdot e_{i_1}\ot \cdots \ot e_{i_r}, e_{i_1}\ot \cdots \ot e_{i_r} \rangle\\
&=\langle \chi^{V_1\ot \cdots \ot V_r}, 1_{\Sym_{\bm\mu}}  \rangle\\
&=1.
\eal
\eeq
Hence for $1\leq k \leq s$
\beq
\sqrt{\frac{h(\lambda)}{\alpha(I)}}\mathcal{E}_{\Tc_k}\cdot e_{i_1}\ot \cdots \ot e_{i_r}=a_k \zeta_{\theta_{\bm\mu}(\Tc_1)}\ot v_{\Tc_k},
\eeq
where $\sum_{k=1}^s ||a_k||^2=1$.

\end{proof}

\begin{proof}[\textbf{Proof of theorem \ref{super-Kostant-thm}}]
By the definition of super-immanants,
the left-hand side of \eqref{super-Kostant-iden} equals to
\beq
\bal
\frac{\Imm_{\chi^{\lambda}}(X_I)}{\alpha(I)}
&=\frac{(-1)^{\bar I}}{\alpha(I)}\langle i_{1},\ldots ,i_{r}\mid \chi^{\lambda}X_1\cdots X_{r} \mid i_{1},\ldots,i_{r}\rangle\\
&=\frac{(-1)^{\bar I}}{\alpha(I)}(\chi^{\lambda}X_1\cdots X_{r})_{i_{1},\ldots,i_{r}}^{i_{1},\ldots,i_{r}}\\
&=\frac{(-1)^{\bar I}}{\alpha(I)}\sum_{\tau\in \Sym_r} \chi^{\lambda}(\tau) (P_{\tau})_{i_{\tau(1)},\ldots,i_{\tau(r)}}^{i_{1},\ldots,i_{r}} x_{i_{\tau(1)},i_1} \cdots x_{i_{\tau(r)},i_r}.
\eal
\eeq

According to Lemma \ref{schur-weyl-corr},   the right-hand side of \eqref{super-Kostant-iden} equals to
\begin{align*}
&str(\mathcal P_{\bm\mu}\ot 1)\circ\Delta\circ \mathcal P_{\bm\mu}|_{L(\bm\lambda)}\\
&=(-1)^{\bar I}\langle  \sqrt{\frac{h(\lambda)}{\alpha(I)}}(\mathcal P_{\bm\mu}\ot 1) \circ \Delta \circ\sum_{\Tc}\mathcal{E}_{\Tc}\cdot e_{i_1}\ot \cdots \ot e_{i_r}, \\
& \quad \sqrt{\frac{h(\lambda)}{\alpha(I)}}\sum_{\Tc}\mathcal{E}_{\Tc}\cdot e_{i_1}\ot \cdots \ot e_{i_r}\rangle\\
&=\frac{(-1)^{\bar I}h(\lambda)}{\alpha(I)}\langle  (\mathcal P_{\bm\mu}\ot 1) \circ \Delta \circ\sum_{\Tc}\mathcal{E}_{\Tc}\cdot e_{i_1}\ot \cdots \ot e_{i_r}, e_{i_1}\ot \cdots \ot e_{i_r}\rangle.
\end{align*}

By the relation of characters and primitive idempotents of the symmetric group, one has that
\beq
\bal
&\frac{(-1)^{\bar I}h(\lambda)}{\alpha(I)}  (\mathcal P_{\bm\mu}\ot 1) \circ  \Delta \circ\sum_{\Tc}\mathcal{E}_{\Tc}\cdot e_{i_1}\ot \cdots \ot e_{i_r}\\
&=\frac{(-1)^{\bar I}}{\alpha(I)} (\mathcal P_{\bm\mu}\ot 1)\circ \Delta \circ\sum_{\si\in \Sym_r} \chi^{\lambda}(\si) P_{\si} \cdot e_{i_1}\ot \cdots \ot e_{i_r}\\
&=\frac{(-1)^{\bar I}}{\alpha(I)} \sum_{\si\in \Sym_r} \chi^{\lambda}(\si) P_{\si}\circ (\mathcal P_{\bm\mu}\ot 1)\sum_{(j_1,\ldots,j_r)} (-1)^{\sum\limits_{a<b}\bar j_a(\bar i_b+\bar j_b)} x_{j_1,i_1}\cdots x_{j_r,i_r}\\
&\quad \ot e_{j_1}\ot \cdots \ot e_{j_r}\\
&=\frac{(-1)^{\bar I}}{\alpha(I)} \sum_{\si,\tau\in \Sym_r} (-1)^{\sum\limits_{a<b}\bar i_{\tau(a)}(\bar i_b+\bar i_{\tau(b)})} \chi^{\lambda}(\si)  x_{i_{\tau(1)},i_1}\cdots x_{i_{\tau(r)},i_r}\\
&\quad \ot P_{\si}\cdot e_{i_{\tau(1)}}\ot \cdots \ot e_{i_{\tau(r)}}.
\eal
\eeq
So taking the  bilinear form with $e_{i_1}\ot \cdots \ot e_{i_r}$, these terms in the above equation are zero unless $\si=\tau$. Then
\beq
\bal
&str(\mathcal P_{\bm\mu}\ot 1)\circ \Delta\circ \mathcal P_{\bm\mu}|_{L(\bm\lambda)}\\
&=\frac{(-1)^{\bar I}}{\alpha(I)} \sum_{\tau\in \Sym_r} (-1)^{\overline{\tau^{-1}(I)}_{\tau} +\sum\limits_{a<b}\bar i_{\tau(a)}(\bar i_b+\bar i_{\tau(b)})} \chi^{\lambda}(\tau)  x_{i_{\tau(1)},i_1}\cdots x_{i_{\tau(r)},i_r}.
\eal
\eeq
Since
\ben
\bal
& P_{\tau}\cdot e_{i_{\tau(1)}}\ot \cdots \ot e_{i_{\tau(r)}}\\
&=(-1)^{\overline{\tau^{-1}(I)}_{\tau}}e_{i_1}\ot \cdots \ot e_{i_r}\\
&=(-1)^{\sum\limits_{a<b}\bar i_{\tau(a)}(\bar i_b+\bar i_{\tau(b)})}(P_{\tau})_{i_{\tau(1)},\ldots,i_{\tau(r)}}^{i_{1},\ldots,i_{r}}
e_{i_1}\ot \cdots \ot e_{i_r},
\eal
\een
we obtain that
\beq
\bal
&str(\mathcal P_{\bm\mu}\ot 1)\circ\Delta\circ \mathcal P_{\bm\mu}|_{L([\Lambda^{\lambda}])}\\
&=\frac{(-1)^{\bar I}}{\alpha(I)}\sum_{\tau\in \Sym_r} \chi^{\lambda}(\tau) (P_{\tau})_{i_{\tau(1)},\ldots,i_{\tau(r)}}^{i_{1},\ldots,i_{r}} x_{i_{\tau(1)},i_1} \cdots x_{i_{\tau(r)},i_r}\\
&=\frac{\Imm_{\chi^{\lambda}}(X_I)}{\alpha(I)}.
\eal
\eeq
\end{proof}

\section{Super Littlewood correspondences}
\subsection{Super Littlewood correspondences I and II}
We establish the relations between super-immanants and  Schur supersymmetric polynomials, which can be viewed as the super analogs of the Littlewood correspondences \cite{Li}.

We first recall some definitions of supersymmetric polynomials.

Let $\mathcal{X}_m = (x_1, \ldots, x_m)$, $\mathcal{Y}_n = (y_1, \ldots, y_n)$ be two sets of indeterminates, the symmetric group $\Sym_m\times \Sym_n$ acts on the polynomial ring $\mathbb C[\mathcal{X}_m, \mathcal{Y}_n] = \mathbb C[x_1, \ldots, x_m, y_1, \ldots, y_n]$. For $f \in \mathbb C[\mathcal{X}_m, \mathcal{Y}_n]$ and $t \in \CC$, we write $f(x_1 = t, y_1 = -t)$ for the polynomial obtained by substituting $x_1 = -y_1 = t$ in $f$.

An $\Sym_m\times \Sym_n$-invariant polynomial $f$ is \textit{supersymmetric} if $f(x_1 = t, y_1 = -t)$ is a polynomial which is independent of $t$. The set of all supersymmetric polynomials is called the  algebra of supersymmetric polynomials in $\mathcal{X}_m, \mathcal{Y}_n$. We will denote it by $\textbf{Sym}^{(m|n)}$.

The rth power sum  supersymmetric polynomial is defined as
\beq
p_{m,n}^{(r)}(\mathcal{X}_m, \mathcal{Y}_n)=(x_1^r+\ldots+x_m^r)+(-1)^{r-1}(y_1^r+\ldots+y_n^r).
\eeq
The Schur supersymmetric polynomials can be defined by different ways, for example see  \cite{Pra,PT,Rem,Macdonald,Mus}.
First we define polynomials $\mathbb{S}_k=\mathbb{S}_k(\mathcal{X}_m, \mathcal{Y}_n)$ by means of the generating function
\beq
S_{m,n}(t)=\prod_{i=1}^m(1-x_it)^{-1}\prod_{j=1}^n(1+y_jt)=\sum_{k \geq 0}\mathbb S_k(\mathcal{X}_m, \mathcal{Y}_n)t^k.
\eeq
Note that if $n=0$, then $\mathbb S_k(\mathcal{X}_m, \mathcal{Y}_n)$ are the homogeneous symmetric polynomials $h_k(\mathcal{X}_m)$, while if $m=0$, then $\mathbb S_k(\mathcal{X}_m, \mathcal{Y}_n)$ are the elementary symmetric polynomials $e_k(\mathcal{Y}_n)$.

Then for any partition $\lambda$ of length $l$, we define the \textit{Schur supersymmetric polynomial} $\mathbb S_{\lambda}= \mathbb S_{\lambda}(\mathcal{X}_m, \mathcal{Y}_n)$ to be the determinant of Jacobi-Trudi matrix $(\mathbb S_{\lambda_i-i+j})_{1\leq i,j\leq l}$
\beq
\mathbb S_{\lambda}=\det(\mathbb S_{\lambda_i-i+j}).
\eeq

We have the super analogs of the Littlewood correspondences I and II \cite{Li}[p. 118, p. 120].
\begin{theorem}\label{LittlewoodI}
   Corresponding to any relation between Schur supersymmetric polynomials of total order $m+n$, we may replace each Schur supersymmetric polynomial by the corresponding super-immanant of complementary principal minors of generator matrix of $A(\Mat_{m|n})$ provided that every product is summed for all sequences of pairwise disjoint subsets of $[m+n]$.
\end{theorem}
In general,
\begin{theorem}\label{LittlewoodII}
   Corresponding to any relation between Schur supersymmetric polynomials, we may replace each Schur supersymmetric polynomial by the corresponding (normalized) super-immanant of a principal minor of generator matrix of $A(\Mat_{m|n})$ provided that we sum for all principal minors of the appropriate order with non-decreasing ordered multisets of $[m+n]$.
\end{theorem}
We only prove Theorem \ref{LittlewoodII}, Theorem \ref{LittlewoodI} can be viewed as a special case of  Theorem \ref{LittlewoodII}.
\begin{proof}[Proof of Theorem \ref{LittlewoodII} ]
  We only need to prove the case that the product of two Schur supersymmetric polynomials, suppose that $|\mu|+|\nu|=k$  and $$\mathbb{S}_{\mu}\mathbb{S}_{\nu}=\sum_{\lambda \vdash k}c_{\mu\nu}^{\lambda}\mathbb{S}_{\lambda}.$$
Then we consider the induced representation of tensor product of irreducible representations $V^{\mu},V^{\nu}$, from the relation of characters and primitive idempotents,  we have that
\begin{align*}
\chi^{\mathrm{Ind}(V^{\mu}\ot V^{\nu})}=\sum_{\si \in \mathfrak{S}_{k}}\si\mathcal{E}^{\mu}\mathcal{E}^{\nu}\si^{-1}=\sum_{\lambda}c_{\mu\nu}^{\lambda}\chi^{\lambda}.
\end{align*}
Hence it follows from proposition \ref{imm-char} that  for any non-decreasing ordered multisets $I=(1\leq i_1\leq \ldots \leq i_{k} \leq m+n)$
\begin{equation}
\begin{aligned}
&\Imm_{\chi^{\mathrm{Ind}(V^{\mu}\ot V^{\nu})}}(X_I) \\
   &=(-1)^{\bar I}\langle i_1,\ldots ,i_k
\mid \chi^{\mathrm{Ind} (V^{\mu}\ot V^{\nu})} X_1\cdots X_{k} \mid i_1,\ldots ,i_k\rangle\\
   &=(-1)^{\bar I}\sum_{\sigma\in\mathfrak{S}_k}\langle i_{\sigma_1},\dots ,i_{\sigma_k}\mid \mathcal{E}^{\mu} \mathcal{E}^{\nu}X_1\cdots X_k \mid i_{\sigma_1},\dots,i_{\sigma_k}\rangle.
\end{aligned}
\end{equation}

For each sequence $(I_1,I_2)$ of  non-decreasing ordered multisets of $[m+n]$ satisfying $|I_1|=|\mu|, \ |I_2|=|\nu|$ and the disjoint union of $I_1$ and $I_2$ is $I$,  there exists the minimal length coset representative elements $\si \in \mathcal M(\mathfrak{S}_k/\Sym_{|\mu|}\times \Sym_{|\nu|})$ such that $\si^{-1}(I)= (I_1,I_2)$.
Hence by proposition  \ref{imm-char}  we have that
\begin{equation*}
\begin{aligned}
&\Imm_{\chi^{\mathrm{Ind}(V^{\mu}\ot V^{\nu})}}(X_I) \\
&=(-1)^{\bar I}\sum_{ \tau \in \Sym_{|\mu|}\times \Sym_{|\nu|}, \atop \si \in \mathcal M(\mathfrak{S}_k/\Sym_{|\mu|}\times \Sym_{|\nu|}) }\langle i_{\tau_{\si_1}},\dots ,i_{\tau_{\si_k}}\mid \mathcal{E}^{\mu} \mathcal{E}^{\nu} X_1 \cdots X_k \mid i_{\tau_{\si_1}},\dots ,i_{\tau_{\si_k}}\rangle\\
&=\sum_{(I_1,I_2)}\frac{\alpha(I)}{\alpha(I_1)\alpha(I_2)}{\Imm}_{\chi^{\mu}}(X_{I_1}) {\Imm}_{\chi^{\nu}}(X_{I_2}),
\end{aligned}
\end{equation*}
where the sums are over all sequences $(I_1,I_2)$ of  non-decreasing ordered multisets of $[m+n]$  and the disjoint union of $I_j, \ j=1,2$ is $I$.

On the other hand,
\begin{equation}
\begin{aligned}
\Imm_{\chi^{\mathrm{Ind}(V^{\mu}\ot V^{\nu})}}(X_I)
   &=(-1)^{\bar I}\langle i_1,\ldots ,i_k
\mid \chi^{\mathrm{Ind}(V^{\mu}\ot V^{\nu})} X_1\cdots X_{k} \mid i_1,\ldots ,i_k \rangle\\
   &=\sum_{\lambda}c_{\mu\nu}^{\lambda}(-1)^{\bar I}\langle i_1,\ldots ,i_k  \mid \chi^{\lambda} X_1\cdots X_{k} \mid i_1,\ldots ,i_k \rangle\\
   &=\sum_{\lambda}c_{\mu\nu}^{\lambda}\Imm_{\chi^{\lambda}}(X_I).
\end{aligned}
\end{equation}
Hence,
\[
\sum_{(I_1,I_2)}\frac{\Imm_{\chi^{\mu}}(X_{I_1}) \Imm_{\chi^{\nu}}(X_{I_2})}{\alpha(I_1)\alpha(I_2)}  =\sum_{\lambda}c_{\mu\nu}^{\lambda}\frac{ \Imm_{\chi^{\lambda}}(X_I)}{\alpha(I)}.
\]
\end{proof}

In particular, we have general Littlewood-Merris-Watkins identities \cite{KS,MW} for supermatrices.
Let  $\psi^{\mu}$ (resp. $\phi^{\mu}$) be the induced characters of the sign character (resp. trivial character) of the parabolic subgroup $\Sym_{\mu}$ of $\Sym_k$.
They can be decomposeed into  the irreducible $\Sym_k$-characters:
\begin{align*}
 \psi^{\mu}=\sum_{\lambda}K_{\lambda^{T},\mu}\chi^{\lambda}, \ \ \
\phi^{\mu}=\sum_{\lambda}K_{\lambda,\mu}\chi^{\lambda},
\end{align*}
where $K_{\lambda,\mu}$ are the Kostka numbers.
\begin{corollary}
   Let $I=(1\leq i_1\leq \ldots \leq i_k \leq m+n)$ be a multiset of $[m+n]$. Fix a partition $\lambda=(\lambda_1,\dots,\lambda_l)$ of $k$. Then
\begin{equation}\label{gene-qLMW}
\begin{aligned}
   &\Imm_{\psi^{\lambda}}(X_I)= \sum_{(I_1,\dots,I_l)}\frac{ \alpha(I)}{ \alpha(I_1)\cdots \alpha(I_l)}\Imm_{\chi^{(1^{\lambda_1})}}(X_{I_1})\cdots \Imm_{\chi^{(1^{\lambda_l})}}(X_{I_l}),\\
   &\Imm_{\phi^{\lambda}}(X_I)=\sum_{(I_1,\dots,I_l)}\frac{ \alpha(I)}{ \alpha(I_1)\cdots \alpha(I_l)}\Imm_{\chi^{(\lambda_1)}}(X_{I_1})\cdots \Imm_{\chi^{(\lambda_l)}}(X_{I_l}),
\end{aligned}
\end{equation}
where the sums are taken over all sequences $(I_1,\dots,I_l)$ of non-decreasing ordered multisets of $[m+n]$ satisfying $|I_j|=\lambda_j$ and the disjoint union of $I_j, 1 \leq j \leq l$ is $I$.
\end{corollary}

\subsection{MacMahon Master Theorem}
Define $\alpha_0=\beta_0=1$, for $k<0$ $\alpha_k=\beta_k=0$ and for $k>0$,
\begin{equation}\label{genefcn1}
\begin{aligned}
  &\alpha_k=str_{1,\ldots,k}\Ec^{(1^{k})}X_1\cdots X_k=\sum_{|I|=k}\frac{\Imm_{\chi^{(1^{k})}}(X_I)}{\alpha(I)},\\ &\beta_k=str_{1,\ldots,k}\Ec^{(k)}X_1\cdots X_k=\sum_{|I|=k}\frac{\Imm_{\chi^{(k)}}(X_I)}{\alpha(I)},
\end{aligned}
\end{equation}
where the sums are taken over all non-decreasing multisets $I$ with $k$ elements. Note that $\{\alpha_k \mid k \in \mathbb{Z}^{+}\}$ in $A(\Mat_{m|n})$  pairwise commute.
The  commutative subalgebra $\mathfrak{B}_{m|n}$ of $A(\Mat_{m|n})$ is generated by $\{\alpha_k,\mid  k \in \mathbb{Z}^{+}\}$.

We set
\begin{align}
\lambda(t)=\sum_{k=0}^{\infty}t^k \alpha_k,\qquad
\si(t)=\sum_{k=0}^{\infty}t^k \beta_k.
\end{align}
The following can be viewed as the super version of the MacMahon Master Theorem in \cite{MRa}.
\begin{theorem}\label{Mac}
$\lambda(-t)\times \si(t) =1.$
\end{theorem}

Let $Y, Z\in \Mat_{m|n}(R)$. Denote $Y*Z=str_1 P Y_1Z_2$.
Let $X^{[k]}$ be the $k$th power of $X$ under the multiplication $*$, i.e.
\begin{align}\label{e:power-m}
X^{[0]}=1,\   X^{[1]}=X, \ X^{[k]}=X^{[k-1]}*X,\ k>1.
\end{align}
We denote
\begin{align}
\psi(t)=\sum_{k=0}^{\infty}t^k\gamma_{k+1},
\end{align}
where $\gamma_k=strX^{[k]}$.

The following Newton identities follow from the super
MacMahon Master Theorem, see \cite{MRa}.
\begin{theorem}[Newton's identities]\label{Newton-iden}
\begin{align}
&\partial_t \lambda(-t)=-\lambda(-t) \psi(t),\\
&\partial_t \si(t)= \psi(t)\si(t).
 \end{align}
\end{theorem}

\subsection{Super Goulden-Jackson identities and super Littlewood correspondence III }

Fix a partition $\lambda \vdash r$. On the  commutative algebra $\mathfrak{B}_{m|n}$ generated by $\alpha_i$ or $\beta_i$ \eqref{genefcn1}, we denote
\[
A=(\alpha_{\lambda_i^{T}-i+j})_{\lambda_1\times \lambda_1},\ \ \ \ \  B=(\beta_{\lambda_i-i+j})_{\lambda_1^T\times \lambda_1^T}.
\]
Then we have the following two special elements in $\mathfrak{B}_{m|n}$:
\begin{align*}
  &\det(A)=\sum_{\mu}K_{\lambda^{T},\mu}^{-1}\alpha_{\mu_1}\cdots \alpha_{\mu_{\lambda_1}}, \\
  &\det(B)=\sum_{\mu}K_{\lambda,\mu}^{-1}\beta_{\mu_1}\cdots \beta_{\mu_{\lambda_1}}.
\end{align*}

The following is the generalization in  super case  of the  Goulden-Jackson identities that appear as Theorem 2.1 in \cite{GJ} and
Theorem 3.2 in \cite{KS}.  It can be viewed as the generalization of the Jacobi-Trudi identity of Schur supersymmetric polynomials.
\begin{theorem}\label{quantum-JT} We have that
\begin{equation}\label{GJ-identity}
\det(A)=\det(B)=str_{1,\ldots,r}(\mathcal{E}_{\Tc}^{\lambda}X_1\cdots X_r) =\sum_{I}\frac{ \Imm_{\chi^{\lambda}}(X_I)}{\alpha(I)},
\end{equation}
where the sum is over all non-decreasing ordered multisets $I$ of $[m+n]$ satisfying $|I|=r$.
\end{theorem}
\begin{proof}
  By definition,
  \begin{equation*}
  \begin{aligned}
     \det(A)
     &=\sum_{\mu\vdash r} K_{\lambda^{T},\mu}^{-1}str_{1,\ldots,r}
     \mathcal{E}^{(1^{\mu_1})}\cdots \mathcal{E}^{(1^{\mu_{\lambda_1}})}X_1\cdots X_r\\
     &=\sum_{\mu\vdash r}\sum_{J}(-1)^{\bar J}\langle j_1,\dots ,j_r \mid K_{\lambda^{T},\mu}^{-1}\mathcal{E}^{(1^{\mu_1})}\cdots \mathcal{E}^{(1^{\mu_{\lambda_1}})}X_1\cdots X_r\mid  j_1,\dots ,j_r  \rangle,
  \end{aligned}
  \end{equation*}
and
\ben
\bal
&str_{1,\ldots,r}(\mathcal{E}_{\Tc}^{\lambda}X_1\cdots X_r)\\
&= \sum_{J}(-1)^{\bar J} \langle  j_1,\dots ,j_r  \mid \mathcal{E}_{\Tc}^{\lambda}X_1\cdots X_r \mid  j_1,\dots ,j_r  \rangle\\
&= \sum_{I}\frac{(-1)^{\bar I}}{\alpha(I)}\sum_{\si \in \Sym_r} \langle  i_{\si_1},\dots ,i_{\si_r} \mid \mathcal{E}_{\Tc}^{\lambda}X_1\cdots X_r \mid  i_{\si_1},\dots ,i_{\si_r} \rangle\\
&=\sum_{I}\frac{\Imm_{\chi^{\lambda}}(X_I)}{\alpha(I)},
\eal
\een
  where the sum in the first line is over all sequences $J$ with $r$ elements and the sum in the second line runs over all non-decreasing multisets $I$ of $[m+n]$ satisfying $|I|=r$.

  By the  Proposition  \ref{imm-char}, we have that
   \begin{equation*}
  \begin{aligned}
     \det(A)&=\sum_{\mu}\sum_{J}(-1)^{\bar J} \langle  j_1,\dots ,j_r  \mid K_{\lambda^{T},\mu}^{-1}\mathcal{E}^{(1^{\mu_1})}\cdots \mathcal{E}^{(1^{\mu_{\lambda_1}})}X_1\cdots X_r\mid  j_1,\dots ,j_r  \rangle\\
     &=\sum_{\mu}\sum_{J'}\frac{(-1)^{\bar J'}} {\alpha(J')}\langle  j'_1,\dots ,j'_r  \mid K_{\lambda^{T},\mu}^{-1} \sum_{\si\in \mathfrak{S}_r}P_{\si}\mathcal{E}^{(1^{\mu_1})}\cdots \\
     &\quad \cdots \mathcal{E}^{(1^{\mu_{\lambda_1}})}P_{\si^{-1}} X_1\cdots X_r\mid j'_1,\dots ,j'_r \rangle\\
     &=\sum_{\mu}\sum_{J'}\frac{(-1)^{\bar J'}}{\alpha(J')}\langle j'_1,\dots ,j'_r \mid K_{\lambda^{T},\mu}^{-1} \psi^{\mu} X_1\cdots X_r\mid  j'_1,\dots ,j'_r \rangle\\
     &=\sum_{J'}\frac{(-1)^{\bar J'}}{\alpha(J')}\langle  j'_1,\dots ,j'_r \mid\chi^{\lambda} X_1\cdots X_r\mid  j'_1,\dots ,j'_r\rangle\\
     &=\sum_{J'}\frac{\Imm_{\chi^{\lambda}}(X_{J'})}{\alpha(J')},
  \end{aligned}
  \end{equation*}
  where the second sum of the second line runs over all non-decreasing multisets $J'=(j'_1\leq \dots \leq j'_r)$ of $[m+n]$, and the fourth line is because
   $\sum_{\mu}K_{\lambda^{T},\mu}^{-1}\psi^{\mu}=\chi^{\lambda}$. We can follow $\det(B)=str_{1,\ldots,r}\mathcal{E}^{\lambda} X_1\cdots X_r$ by the same method.
\end{proof}

Let the generator supermatrix  $X=(x_{ij})$ of $A(\Mat_{m|n})$ have the block form
\[
X = \begin{pmatrix}
A & B \\
C & D
\end{pmatrix},
\]
where submatrices $A,B,C,D$ are, respectively, $m \times m$,  $m \times n$,  $n \times m$,  $n \times n$ dimensional matrices. The submatrices $A$, $D$ and $B$, $C$, respectively, have even and odd parity.

 We consider the solutions of characteristic polynomials $a(t)=\det(tI-A)$ and $d(t)=\det(tI-D)$ over the algebraic closure of fraction field of domain generated by field of  $A_{ij},D_{kl},1 \leq i,j\leq m, 1 \leq k,l\leq n$, and denote the characteristic roots respectively  by $a_1,\ldots,a_m$ and $d_1,\ldots,d_n$. We have the  following lemma for diagonalization of supermatrix $X$  in \cite{KN88}.
\begin{lemma}\label{diag-lem}
Under the assumption above, there exists an invertible supermatrix $U$ such that
\beq\label{UMU-1}
U^{-1}XU= \diag(\omega_1,\ldots,\omega_m, \varpi_1,\ldots,\varpi_n),
\eeq
where $\omega_i=a_i+\mu_i$, $\varpi_j=d_j+\nu_j$ such that $\mu_i^2=\nu_j^2=0$.
\end{lemma}
\begin{proof}
Firstly, there are invertible matrices $V_1$ and $V_2$ with even entries such that $V_1^{-1}AV_1=\diag(a_1,\ldots,a_m)$ and $V_2^{-1}DV_2=\diag(d_1,\ldots,d_n)$. Let $V=\begin{pmatrix}
V_1 & 0 \\
0 & V_2
\end{pmatrix}$ and $X'=V^{-1}XV=\begin{pmatrix}
A' & B' \\
C' & D'
\end{pmatrix}$. Then we claim that there exists  eigenvalues $\omega_i=a_i+\mu_i$ and $\varpi_j=d_j+\nu_j$ of $X'$ such that
\beq
\bal
X'z_i=\omega_i z_i,\ \ X'w_j=\varpi_j w_j,
\eal
\eeq
where $\mu_i^2=\nu_j^2=0$ and
$z_i=(x_1,\ldots,x_{i-1},1,x_{i+1},\ldots,x_m,y_1,\ldots,y_n)^T$ (resp. $w_j=(p_1,\ldots,p_{m},q_1,\ldots,q_{j-1},1,q_{j+1},\ldots,q_n)^T$)
is called even (resp. odd) eigenvector in \cite{KN88}.

We only consider one equation $X'z_1=(a_1+\mu_1)z_1$ and can get other eigenvalues and their corresponding eigenvectors by the same method. From the equation $X'z_1=(a_1+\mu_1)z_1$, we have the following system of equations:
\begin{equation}\label{eq-sys}
     \begin{cases}
         B'_1y=\mu_1, \\
         a_2x_2+B'_2y=(a_1+\mu_1)x_2,\\
         \cdots \cdots\\
         a_mx_m+B'_my=(a_1+\mu_1)x_m, \\
         C'_1x+d_1y_1=(a_1+\mu_1)y_1,\\
         \cdots \cdots\\
         C'_nx+d_ny_n=(a_1+\mu_1)y_n,
     \end{cases}
 \end{equation}
where let $B'=\begin{pmatrix}
B'_1 \\ \vdots \\ B'_m
\end{pmatrix}$, $C'=\begin{pmatrix}
C'_1 \\ \vdots \\ C'_n
\end{pmatrix}$, $x=\begin{pmatrix}
1 \\ x_2 \\ \vdots \\ x_m
\end{pmatrix}$, $y=\begin{pmatrix}
 y_1 \\ \vdots \\ y_n
\end{pmatrix}$.
Hence from the first $m$ equations we get
\beq
x_i=-(a_i-a_1-B'_1y)^{-1}B'_i y, \ 2 \leq i \leq m.
\eeq
And by substituting $x_i$  and $\mu_1$  in the $(m+1)$-th equation and continuing so on, we get $y_j$ is expressed as a polynomial in $y_{j+1},\ldots,y_n$ for $1 \leq j \leq n$ by using fact $y_i^2=0$. Hence the system \eqref{eq-sys} of equations is solved. Finally,  let $Q=(z_1,\ldots ,z_m,w_1,\ldots,w_n)$ and $U=VQ$, then  $$U^{-1}XU= \diag(\omega_1,\ldots,\omega_m, \varpi_1,\ldots,\varpi_n).$$
\end{proof}

\begin{proposition}\label{alpha-sym}
For $k>0$,
\beq
\alpha_k=\mathbb S_{(1^k)}(\omega_1,\ldots,\omega_m,-\varpi_1,\ldots,-\varpi_n).
\eeq
\end{proposition}
\begin{proof}
It follows from identity \eqref{UMU-1} with the multiplicative property of Berezinian and definition of Schur supersymmetric polynomials that the characteristic function
(the Berezinian of $(tI -X)$) can be written by supersymmetric polynomials:
\beq
\bal
\Ber(tI-X)&=\frac{\sum_{i=0}^{m}(-1)^i e_i(\omega_1,\ldots,\omega_m)t^{m-i}}
{\sum_{j=0}^{n}(-1)^j e_j (\varpi_1,\ldots,\varpi_n)t^{n-j}}\\
&=\sum_{k=0}^{\infty}(-1)^{k} \mathbb S_{(1^k)}(\omega_1,\ldots,\omega_m,-\varpi_1,\ldots,-\varpi_n)t^{m-n-k}.
\eal
\eeq

On the other hand, the Berezinian $\Ber(tI-M)$ can   be written as
\beq\label{Ber-str}
\Ber(tI-X)=\sum_{k=0}^{\infty}(-1)^{k}\alpha_kt^{m-n-k},
\eeq
see \cite{Ber,KV,Sch}. Generally, Molev and Ragoucy proved the noncommutative analogues for super Manin matrices in \cite{MRa}.
\end{proof}

Similar to the classic case, the super-immanants of generator supermatrix of coordinate superalgebra can be expressed as a Schur supersymmetric polynomials at particular arguments. Explicitely, we have the super analog of Littlewood correspondence III \cite{Li}[p. 121].
\begin{theorem}\label{t:Litt3}
   Let $\lambda \vdash r$ and $\omega_1,\ldots,\omega_m, \varpi_1,\ldots,\varpi_n$ be the eigenvalues of generator supermatrix $M$, then
\begin{equation}\label{Littlewood3}
\mathbb{S}_{\lambda}(\omega_1,\ldots,\omega_m,-\varpi_1,\ldots,-\varpi_n)
=str_{1,\ldots,r}\Ec^{\lambda} X_1\ldots X_r= \sum_{|I|=r} \frac{\Imm_{\chi^{\lambda}}(X_I)}{\alpha(I)},
\end{equation}
where the sum is taken over all non-decreasing ordered multisets $I$ of $[m+n]$  satisfying $|I|=r$.
In particular,
\beq
\bal
&\beta_r=\mathbb{S}_{r}(\omega_1,\ldots,\omega_m,-\varpi_1,\ldots,-\varpi_n),\\
&\alpha_r=\mathbb{S}_{(1^r)}(\omega_1,\ldots,\omega_m,-\varpi_1,\ldots,-\varpi_n),\\
&\gamma_r=strX^{[r]}=p_{m,n}^{(r)}(\omega_1,\ldots,\omega_m,-\varpi_1,\ldots,-\varpi_n).
\eal
\eeq
\end{theorem}
\begin{proof}
From proposition \ref{alpha-sym},  we have that $$\alpha_r=\mathbb{S}_{(1^r)}(\omega_1,\ldots,\omega_m,-\varpi_1,\ldots,-\varpi_n).$$
Hence, we can reduce \eqref{Littlewood3} by the super-Goulden-Jackson identities \eqref{GJ-identity} and the definition of Schur supersymmetric polynomials:
\beq
\bal
&\mathbb{S}_{\lambda}(\omega_1,\ldots,\omega_m,-\varpi_1,\ldots,-\varpi_n)\\
&=\sum_{\mu}K_{\lambda,\mu}^{-1}\beta_{\mu_1}\cdots \beta_{\mu_{\lambda_1}}\\
&=\sum_{\mu}K_{\lambda^{T},\mu}^{-1}\alpha_{\mu_1}\cdots \alpha_{\mu_{\lambda_1}}\\
&=\sum_{I}\frac{ \Imm_{\chi^{\lambda}}(X_I)}{\alpha(I)}.
\eal
\eeq

Finally, by Newton's identities in theorem \ref{Newton-iden}, $\gamma_r=strX^{[r]}$ can be expressed as the same linear combination of $\alpha_{\mu}, \mu \vdash r$ as $p_{m,n}^{(r)}$.
\end{proof}
\begin{remark}
The Littlewood correspondence III also holds for these  separable supermatrices defined in \cite{KN84,KN88}. In fact, any separable supermatrix $X$  can be diagonalized by the same argument in Lemma \ref{diag-lem}.
\end{remark}

Consider the algebra homomorphism $\Phi: A(\Mat_{m|n}) \rightarrow \mathbb{C}[\mathcal{X}_m, \mathcal{Y}_n]$ defined as
\ben
\bal
&x_{ij}\mapsto 0, \quad i\neq j, \\
&x_{ii}\mapsto x_i, \quad 1 \leq i \leq m,\\
&x_{jj}\mapsto -y_j, \quad  1 \leq j \leq n.
\eal
\een
\begin{theorem}\label{iso-sym} The restriction of $\Phi$:
  $$\Phi |_{\mathfrak{B}_{m|n}} :\mathfrak{B}_{m|n}\rightarrow \textbf{Sym}^{(m|n)}$$
  is an algebra isomorphism. And
  $\{ \sum_{I} \frac{ \Imm_{\chi^{\lambda}}(X_I)}{\alpha(I)}\mid \lambda \in H(m,n)\}$ is a basis of $\mathfrak{B}_{m|n}$.
\end{theorem}

\begin{proof}
It is easy to see that $\Phi(\alpha_k)=\mathbb{S}_{(1^k)}(\mathcal{X}_m, -\mathcal{Y}_n)$. So the map $\Phi |_{\mathfrak{B}_{m|n}}$ is an epimorphism by the fundamental theorem of supersymmetric polynomials.
Because of  the super Littlewood correspondence III, $\Phi|_{\mathfrak{B}_{m|n}}$ is an isomorphism. So $\{ \sum_{I} \frac{ \Imm_{\chi^{\lambda}}(X_I)}{\alpha(I)}\mid \lambda \in H(m,n)\}$ is a basis of $\mathfrak{B}_{m|n}$.
\end{proof}

The following theorem can be viewed as a generalization of the relations between  Schur polynomials and power-sum polynomials in \cite{Li}.
\begin{theorem}\label{rela-SchurPower}
Let $\lambda \vdash r$, then
$$\sum_{|I|=r} \frac{ \Imm_{\chi^{\lambda}}(X_I)}{\alpha(I)}=\frac{\Imm_{\chi^{\lambda}} (\Gamma_r)}{r!},$$
where the immanants on the  right hand side are the classical immanants for non-super matrices  and the lower Hessenberg matrix  $\Gamma_r$ is
$$
\Gamma_r=\begin{pmatrix}
\gamma_1 & 1 &  & &  \\
\gamma_2 & \gamma_1 & 2 &  & \\
\gamma_3 & \gamma_2 & \gamma_1 & 3 & \\
\vdots & \vdots & \ddots & \ddots &  \ddots\\
\gamma_r & \gamma_{r-1} & \cdots & \cdots & \gamma_1
\end{pmatrix}.$$
\end{theorem}
\begin{proof}
By the Newton identities in Theorem \ref{Newton-iden} and Cramer's rule, we have that
\[
\alpha_{r}=\frac{\det(\Gamma_r)}{r!}.
\]
Moreover, the Schur supersymmetric polynomials can be written as
\[
r!\ \mathbb{S}_{\lambda}=\Imm_{\chi^{\lambda}}(\mathfrak{P}_r),
\]
where the matrix $\mathfrak{P}_r$ is  obtained by replacing  $\gamma_i$ with the power sum supersymmetric  polynomial $p_{m|n}^{(i)}$ in matrix $\Gamma_r$. According to the isomorphism $\Phi|_{\mathfrak{B}_{m|n}}$ and  the super-Goulden-Jackson identities, we have that
\ben
\bal
\sum_{|I|=r} \frac{ \Imm_{\chi^{\lambda}}(X_I)}{\alpha(I)}=
\sum_{\mu}K_{\lambda^{T},\mu}^{-1}\alpha_{\mu_1}\cdots \alpha_{\mu_{\lambda_1}}
=\frac{\Imm_{\chi^{\lambda}}(\Gamma_r)}{r!}.
\eal
\een
\end{proof}

%

\bibliographystyle{amsalpha}

\end{document}